\documentclass[12 pt, journal, draftcls, onecolumn]{IEEEtran}
\usepackage{enumerate}
\usepackage{amsmath, amsthm, amssymb}
\usepackage{float}
\usepackage{subcaption}
\usepackage{flexisym}
\usepackage{mathtools}
\usepackage{breqn}
\usepackage{array}
\usepackage{tikz}
\usetikzlibrary{shapes.geometric, arrows}
\usepackage[font = small]{caption}
\usepackage{paralist}
\usepackage{mathrsfs}
\allowdisplaybreaks[1]

\newtheorem{lemma}{Lemma}
\newtheorem{theorem}{Theorem}

\newtheorem{corollary}{Corollary}
\newtheorem{assumption}{Assumption}

\DeclareMathOperator*{\trace}{trace}
\DeclareMathOperator*{\diag}{diag}

\begin{document}
\title{Cyber Physical Attacks with Control Objectives}
\author{Yuan Chen, Soummya Kar, and Jos\'{e} M. F. Moura
\thanks{Yuan Chen {\{(412)-268-7103\}}, Soummya Kar {\{(412)-268-8962\}}, and Jos\'{e} M.F. Moura {\{(412)-268-6341, fax: (412)-268-3890\}} are with the Department of Electrical and Computer Engineering, Carnegie Mellon University, Pittsburgh, PA 15217 {\tt\small \{yuanche1, soummyak, moura\}@andrew.cmu.edu}}
\thanks{This material is based on research sponsored by DARPA under agreement number DARPA FA8750-12-2-0291. The U.S. Government is authorized to reproduce and distribute reprints for Governmental purposes notwithstanding any copyright notation thereon.}
\thanks{The views and conclusions contained herein are those of the authors and should not be interpreted as necessarily representing the official policies or endorsements, either expressed or implied, of DARPA or the U.S. Government.}}
\maketitle

\begin{abstract}
	This paper studies attackers with control objectives against cyber-physical systems (CPS). The system is equipped with its own controller and attack detector, and the goal of the attacker is to move the system to a target state while altering the system's actuator input and sensor output to avoid detection. We formulate a cost function that reflects the attacker's goals, and, using dynamic programming, we show that the optimal attack strategy reduces to a linear feedback of the attacker's state estimate. By changing the parameters of the cost function, we show how an attacker can design optimal attacks to balance the control objective and the detection avoidance objective. Finally, we provide a numerical illustration based on a remotely-controlled helicopter under attack.
\end{abstract}

\section{Introduction}\label{sect: introduction}
The study of cyber-physical security has primarily focused on the analysis of attacker capabilities and the development of resilient algorithms for attack detection, state estimation, and control. Cyber-physical systems (CPS), systems of intercommunicating sensing, computation, and actuation components coupled with an underlying physical process, have been the targets of sophisticated cyber attacks that result in catastrophic physical damage. StuxNet~\cite{Cardenas} and the Maroochy Shire Council Sewage Control Incident~\cite{CardenasOld} are two examples of attacks against industrial systems, and the authors of~\cite{CarAttack} and~\cite{CarAttack2} demonstrate similar security vulnerabilities for commercial automobiles. While the existing literature has analyzed the capability of attackers to bypass security protocols, there is little work on how an attacker should attack a CPS in order to accomplish specific goals.

In this paper, we study attackers who have specific control objectives and wish to evade security protocols employed by the CPS. We consider attackers that perform integrity attacks~\cite{Mo, TeixeiraModels}, in which the attacker alters the input on some of the actuators and the output on some of the sensors, in order to move the system to a target state while avoiding the system's attack detector. The CPS is affected by sensor and process noise, so, for state estimation and control, the system is equipped with a Kalman filter and an LQG controller. One popular method of attack detection is the $\chi^2$ fault detector: the detector examines the energy of the Kalman filter innovations process and reports an attack if the energy exceeds a threshold~\cite{Willsky}. Our work in this paper is to design an attack sequence that balances the energy of the attacker's control objective error and the impact of the attack on the CPS' attack detection statistic. 

We formulate a finite horizon cost function that captures the attacker's objectives. The cost function is the sum of two quadratic terms: one term penalizes the deviation of the system from the target state and the other term penalizes the energy of the detection statistic. The attacker chooses the relative weighting between the control error component and the detection statistic component of the cost function in order to match his or her own motives. Our cost function for modeling the attacker's objectives differs from the standard Linear Quadratic Gaussian (LQG) cost function. In the standard LQG problem, the cost function directly penalizes the magnitude of the input at each time step~\cite{Speyer}. In contrast, our cost function does not directly penalize the magnitude of the attack at each time step and, instead, only accounts for the effect of the attack on the detection statistic. 

This paper addresses two main problems. First, we solve for the optimal attack sequence via dynamic programming. Second, we derive expressions for the control error cost and the detection statistic energy cost of an attack sequence as a function of the weighting parameter between the two components of the cost function. The novelty of our contribution is that we precisely state how an attacker should design an attack in order to move the system to a target state while attempting to avoid detection, a problem that is not addressed by existing literature. A preliminary version of our work appears in~\cite{ChenACC}, where we studied how to attack a CPS that did not have its own controller. Additionally, in this paper, we explicitly show how an attacker can choose the relative weighting between components of the cost function to affect the performance of the optimal attack, and we provide detailed proofs that were omitted in~\cite{ChenACC}. 

Existing literature has examined the ability of an attacker to avoid detection. Reference~\cite{Liu} shows that attacks that only manipulate sensor values are undetectable if they are consistent with the system's sensor placement pattern and if the attack detector does not exploit system dynamics information. When the system is noiseless, references~\cite{Pasqualetti, TeixeiraReveal}, and~\cite{ChenICASSP} relate the ability of an attacker to avoid detection to certain geometric control-theoretic properties of the CPS. In our previous work~\cite{ChenSideInfo}, we provide a necessary and sufficient condition on the extended observability subspaces of the CPS under which an attacker can perfectly hide the effect of his or her attack from the sensor output. Prior work has also analyzed the ability of the attacker to simultaneously cause damage and avoid detection. The authors of~\cite{WeerakkodyTarget} analyze an attacker who can manipulate all sensors and actuators but does not know the entire system model. References~\cite{MoWorkshop} and~\cite{MoWSN} characterize the error in state estimation that a false data injection attacker can cause while evading a residue detector similar to the $\chi^2$ detector. The authors of~\cite{Mo} consider attackers who can attack both the actuators and the sensors of the CPS and characterize the set of reachable states for a detection-evading attacker.

Unlike previous work, which extensively analyzes the attacker's~\textit{ability} to avoid detection and cause damage, the purpose of this paper is to address~\textit{how} an attacker should go about designing an attack to achieve these objectives. We pose an optimal control formulation of the attack design problem -- the attacker's objectives define a linear quadratic cost function that penalizes deviation from the target state and the energy of the detection statistic. In contrast to the standard LQG problem, our cost function does not directly penalize the attack, so, first, we determine the effect of an attack sequence on the detection statistic. Then, through dynamic programming, we calculate the cost-minimizing attack and show that the optimal attack strategy is a linear feedback of the attacker's state estimate. We compute, separately, the control error cost and the detection statistic cost of the optimal attack strategy as a function of the weighting between the two components of the cost function. That is we analyze the trade-offs in the control performance and detection avoidance performance of the optimal attack. The primary consequence of these contributions is that an attacker can select a cost function that best reflects his or her specific goals of control and detection avoidance and design an attack to optimally accomplish those goals. 

The rest of this paper is organized as follows. In Section~\ref{sect: background}, we provide the model of the CPS and attacker, review the $\chi^2$ detector, define the cost function associated with the attacker's goals, and formally state the problems we address. Section~\ref{sect: attackBias} details the effect of an attack sequence on the system's attack detector. In Section~\ref{sect: main}, we solve for the optimal attack sequence using dynamic programming. In Section~\ref{sect: weightedCostFunction} we derive expressions for the control error cost and the detection statistic cost of an optimal attack as a function of the weighting parameter.  We give a numerical example of a remotely-controlled helicopter under attack in Section~\ref{sect: example}, and we conclude in Section~\ref{sect: conclusion}.

\section{Background}\label{sect: background}
\subsection{Notation}\label{sect: notation}
Let $\mathbb{R}$ denote the reals and $\mathbb{R}^n$ denote the space of $n$-dimensional real (column) vectors. The multivariate Gaussian distribution with mean $\mu$ and covariance $\Sigma$ is denoted as $\mathcal{N} \left( \mu, \Sigma \right)$. Let $I_n$ be the $n$ by $n$ identity matrix. For a matrix $M$, $\mathscr{R} (M)$ denotes the range space of $M$, $\mathscr{N} (M)$ denotes the null space of $M$, and $M^\dagger$ denotes the Moore-Penrose pseudoinverse. For a symmetric matrix $S = S^T$, $S \succeq 0$ denotes that $S$ is positive semidefinite, and $S \succ 0$ denotes that $S$ is positive definite.

\subsection{System Model}\label{sect: systemModel}
The CPS follows a discrete time state space model:
\begin{equation}\label{eqn: ssModel1}
\begin{split}
	x_{t+1} &= Ax_{t} + Bu_{t} + \Gamma e_t + w_t, \\
	y_t& = Cx_{t} + \Psi e_t + v_t,
\end{split}
\end{equation}
where $x_t \in \mathbb{R}^n$ is the system state, $u_t \in \mathbb{R}^m$ is the system input, $e_t \in \mathbb{R}^s$ is the attack, $w_t \in \mathbb{R}^n$ is the process noise, $y_t \in \mathbb{R}^p$ is the sensor output, and $v_t \in \mathbb{R}^p$ is the sensor noise. For nonlinear CPS,~\eqref{eqn: ssModel1} represents its linearized dynamics about a nominal operating point. The process noise $w_t$ has independent identical distribution (i.i.d.) $\mathcal{N}\left(0, \Sigma_w\right)$ with $\Sigma_w \succ 0$, the sensor noise $v_t$ has (i.i.d.) distribution $\mathcal{N}\left(0, \Sigma_v\right)$ with $\Sigma_v \succ 0$, and $w_t$ is independent of $v_t$. The system begins running at $t = -\infty$ from a random initial state $x_{-\infty}$ that has distribution $\mathcal{N}\left( \overline{x}_{-\infty}, \Sigma_x\right)$ with $\Sigma_x \succ 0$. The initial state $x_{-\infty}$ is independent of $w_t$ and $v_t$. The pair $(A, C)$ is observable, and the pair $(A, B)$ is controllable. We provide details on $\Gamma$ and $\Psi$, which model the attacker, in Section~\ref{sect: attackerModel}.

\begin{assumption}\label{ass: systemKnowledge}
	The system knows the matrices $A$, $B$, and $C$, the statistics of $x_{-\infty}$, $w_t$, $v_t$, and the input $u_t$ (for all time $t$). The system does not know the matrices $\Gamma$ and $\Psi$ and does not know the attack $e_t$. 
\end{assumption}

The system's objective is to regulate its internal state, $x_t$, to the origin (i.e., $x_t \equiv 0$). Because the system cannot directly observe $x_t$, it constructs an estimate of the state using the observations $y_t$ through a Kalman filter. Then, the system uses its state estimate to calculate its regulating control $u_t$. The system designs its Kalman filter and controller assuming nominal operating conditions (i.e., $e_t = 0$ for all $t$). When $e_t = 0$ for all $t$, the Kalman filter recursively calculates,  $\widehat{x}_{t|t}$, the minimum mean squared error estimate (MMSE) of the state, $x_t$, given all sensor measurements up to time~$t$. The Kalman filter has converged to a fixed Kalman gain since the system starts running at $t=-\infty$: 
\begin{align}\label{eqn: LimitingKalman}
K &= PC^T(CPC^T + \Sigma_v)^{-1},\\
\begin{split}
P &= APA^T + \Sigma_w - APC^T (CPC^T + \Sigma_v)^{-1}CPA^T, \label{eqn: covarianceEquation}
\end{split}\\
\widehat{x}_t & = \widehat{x}_{t|t-1} + K\left(y_{t} - C \widehat{x}_{t|t-1}\right), \label{eqn: filterEquation}\\
\widehat{x}_{t+1|t} &=A\widehat{x}_t + Bu_t. \label{eqn: kalmanPrediction}
\end{align}
The system has a feedback controller
\begin{equation}\label{eqn: sysController}
	u_t = L \widehat{x}_t,
\end{equation}
such that $A+BL$ is stable. Under nominal operating conditions, one such controller is the infinite horizon LQG controller that computes $u_t$ to minimize the cost function
\begin{equation}\label{eqn: sysLQGCost}
	J_{\text{CPS}} = \lim_{T \rightarrow \infty} \frac{1}{2T+1} \mathbb{E} \left[\sum_{t = -T}^{T} x_t^T Q^' x_t + u_t^T R^' u_t\right],
\end{equation}
where $Q^' \succeq 0$, $R^' \succ 0$, and the pair $(A, Q^')$ is observable. For the CPS, the term $x_t^T Q^' x_t$ penalizes the deviation of the state from the origin (i.e., the regulation error), and the term $u_t^T R^' u_t$ represents the cost of using the control $u_t$. The infinite horizon LQG controller aims to minimize the average sum of these two cost terms. The controller that minimizes~\eqref{eqn: sysLQGCost} has the structure of~\eqref{eqn: sysController}, where
\begin{equation}\label{eqn: sysLQG}
	 L = -\left(B^T S B + R^'\right)^{-1}B^T SA,
\end{equation}
and $S$ satisfies
\begin{equation}\label{eqn: sysLQGRiccatti}
\begin{split}
	S &= A^T S A + Q^' -  A^T S B\left(B^T S B + R^'\right)^{-1}B^T SA.
\end{split}
\end{equation}
The CPS represented by~\eqref{eqn: ssModel1} equipped with a Kalman Filter (equations~\eqref{eqn:  LimitingKalman}-\eqref{eqn: kalmanPrediction}) and (stable) feedback controller is a standard model of a CPS~\cite{MoScada, Weerakkody, MoAuthentication}.

Although its Kalman filter and feedback controller are designed assuming nominal operating conditions, the system is equipped with a $\chi^2$ failure detector for detecting operating conditions that are not nominal (i.e., $e_t \neq 0$)~\cite {Willsky, MoScada, MoReplay}. The $\chi^2$ detector examines the innovations sequence of the Kalman filter, $\nu_t$, defined as
\begin{equation}\label{eqn: nuDef}
	\nu_t = y_t - C\widehat{x}_{t|t-1},
\end{equation}
where $\widehat{x}_{t|t-1}$, given by equation~\eqref{eqn: kalmanPrediction}, is, under the condition that $e_t = 0$ for all $t$, the MMSE prediction of $x_t$ using sensor measurements up to $y_{t-1}$. When there is no attack, it is well known~\cite{Speyer} that the innovations are i.i.d. with  $\mathcal{N}\left(0, \Sigma_\nu\right)$, where
\begin{equation}\label{eqn: sigmaNu}
	\Sigma_\nu = CPC^T + \Sigma_v,
\end{equation}
and $\nu_t$ is orthogonal to $\widehat{x}_{t|t-1}$. The $\chi^2$ detector computes the statistic
\begin{equation}\label{eqn: residueNorm}
	g_t = \nu_t^T \Sigma_\nu^{-1} \nu_t,
\end{equation}
and reports that an attack has occurred if $g_t$ exceeds a threshold $\tau$. The threshold $\tau$ is chosen \`{a} priori based on the desired false alarm and missed detection probabilities~\cite{Willsky}. 
\subsection{Attacker Model}\label{sect: attackerModel}
We provide assumptions and details on the attacker model.
\begin{assumption}\label{ass: attackBeginTime}
	The attack begins at time $t = 0$. That is, $e_t \equiv 0$ for all $t < 0$.
\end{assumption}
\noindent Assumption~\ref{ass: attackBeginTime} is made without loss of generality since we can arbitrarily set the indexing of time steps. The attacker can attack the system (i.e., $e_t \neq  0$) over a fixed time window from $t = 0$ to $t = N$. The attacker knows the value of $N$. 
\begin{assumption}\label{ass: attackerSensorKnowledge}
	The attacker causally knows $y_t$ for all times $t$. For time $t=0$ to $t=N$, the attacker also causally knows $\widetilde{y}_t$, where
	\begin{equation}
		\widetilde{y}_t = Cx_t + v_t.
	\end{equation}
\end{assumption}
\noindent Assumption~\ref{ass: attackerSensorKnowledge} states that, at time $t$, the attacker has access to the system sensor measurements for all times up to $t$. 
The attacker also knows the sensor measurement at time $t$ before it is modified by his or her attack at time $t$. The attacker may use its knowledge of $\widetilde{y}_t$ to design his or her attack, $e_t$. 
\begin{assumption}\label{ass: attackerSystemKnowledge}
	The attacker knows the matrices $A$, $B$, $C$ and $L$ and the statistics of $x_{-\infty}$, $w_t$, $v_t$. The attacker knows the matrices $\Gamma$ and $\Psi$.
\end{assumption}
\noindent The attacker knows the system model and the system's feedback matrix $L$. The attacker also knows which sensors and actuators he or she can attack and remembers the history of previous attacks. 
\begin{assumption}\label{ass: GammaPsiInjective}
	The matrix $\left[\begin{array}{c} \Gamma \\ \Psi \end{array}\right]$ is injective.
\end{assumption}
\noindent We make assumption~\ref{ass: GammaPsiInjective} without loss of generality. If the matrix $\left[\begin{array}{c} \Gamma \\ \Psi \end{array}\right]$ is not injective, then we can remove redundant columns from it to form an injective matrix without changing the capabilities of the attacker.

Using knowledge of the system model, sensor measurements $y_t$, and the attacks $e_t$, the attacker performs Kalman filtering to calculate $\widehat{x}_t$, the MMSE estimate of the system state given all the measurements up to time $t$. The attacker's Kalman filter is separate from the system's Kalman filter, but from time $t= - \infty$ to $t = -1$, the two Kalman filters produce the same estimate since $e_t = 0$ in that time interval. Given his or her knowledge, the attacker also calculates the state estimate produced by the \textit{system's} Kalman filter. Since, by assumption~\ref{ass: attackerSystemKnowledge}, the attacker knows the feedback matrix $L$, having knowledge of the system's state estimate means that he or she knows the input $u_t$. Our assumptions state that the attacker knows the system model, sensor outputs, and system LQG control inputs exactly. From the perspective of the system, this represents a worst-case (i.e., most powerful) attacker. The case in which the attacker is not as powerful (e.g., the attacker does not know the system model exactly, the attacker does not know the sensor outputs exactly, etc.) is the subject of future work.

The objective of the attacker is to move the system to a target state $x^*$ while evading the $\chi^2$ detector over a finite time window from $t = 0$ to $t = N$.
\noindent The attacker chooses an attack sequence $e_0, \dots, e_N$ to accomplish his or her goals, and the attack at time $t$, $e_t$, can only depend on the attacker's available information at time $t$. Let $\mathcal{I}_t$ denote the attacker's information set at time $t$. From the assumptions about the attacker's knowledge, we see that $\mathcal{I}_t$ follows the classical information pattern\footnote{The information set $\mathcal{I}_{t}$ is the information of the attacker that is $\textit{in addition}$ to the knowledge he or she has about the system model. For example, even though it is not explicitly stated in the information pattern $\mathcal{I}_t$, at time $t=0$, the attacker knows $\widehat{x}_{t|t-1}$ because of his or her Kalman filter.}~\cite{Speyer}:
\begin{equation}\label{eqn: attackerInformation}
\begin{split}
	\mathcal{I}_0 = \left\{\tilde{y}_0\right\}, \quad \mathcal{I}_{t+1} = \left\{\mathcal{I}_{t}, \tilde{y}_{t+1}, e_{t}\right\}.
\end{split}
\end{equation}
In the time window from $t=0$ to $t=N$, if an attack occurs, i.e., if $e_t \neq 0$ for some $t$, the attacker's Kalman filter produces a different estimate than the system's Kalman filter. This is because, at time $t$, the attacker knows $e_0, \dots, e_t$ and $\widetilde{y}_0, \dots, \widetilde{y}_t$ while the system only knows $y_0, \dots, y_t$. The attacker's Kalman filter becomes
\begin{align}\label{eqn: attackerPredictor}
	\widehat{x}_{t+1|t} &= A\widehat{x}_t + Bu_t + \Gamma e_t, \\
	\widehat{x}_t &= \widehat{x}_{t|t-1} + K\left(\widetilde{y}_t - C\widehat{x}_{t|t-1}\right). \label{eqn: attackerFilter}
\end{align}
The term $\widehat{x}_{t+1|t}$ (equation~\eqref{eqn: attackerPredictor}) is the MMSE estimate of $x_{t+1}$ given information set $\mathcal{I}_t$ and attack $e_t$, and the term $\widehat{x}_t$ (equation~\eqref{eqn: attackerFilter}) is the MMSE estimate of $x_t$ given information set $\mathcal{I}_t$. That is
\begin{align}
	\widehat{x}_{t+1|t} &= \mathbb{E} \left[ x_{t+1} \vert \left\{ \mathcal{I}_t, e_t \right\} \right], \label{eqn: attackerPredictorClarify}\\
	\widehat{x}_t &= \mathbb{E} \left[ x_t \vert \mathcal{I}_t \right]. \label{eqn: attackerFilterClarify}
\end{align}

The attack sequence $e_0, \dots, e_N$ changes the state estimate and the intermediate quantities calculated by the system's Kalman filter (e.g., the innovation $\nu_t$). Let the superscript $^e$ denote the variables associated with the system's Kalman filter under attack, and let the superscript $^0$ denote the variables in the situation that there is no attack (i.e., $e_t = 0$ for $t=0$ to $t=N$). For example, $\nu_t^0$ is the innovation at time $t$ when there is no attack up to time $t$ ($e_0 = e_1 = \dots = e_t = 0$). When the system is under attack, the system's Kalman filter, controller, and $\chi^2$ detector only have access to the variables affected by the attack, i.e., the variables with superscript $^e$. The $\chi^2$ detector calculates the statistic
	$g_t^e = {\nu_t^e}^T \Sigma_\nu^{-1} \nu_t^e$
and compares $g_t^e$ against a threshold to determine whether or not an attack has occurred. 

To capture the attacker's objectives, define the following quadratic cost function:
\begin{equation}\label{eqn: AttackerLQG1}
\begin{split}
	J &= \mathbb{E} \Big[\sum_{t=0}^N \left(x_t-x^*\right)^T Q_t \left(x_t-x^*\right) +  {\nu_t^e}^TR_t\nu_t^e \Big],
\end{split}
\end{equation}
where $Q_t$ and $R_t$ are positive definite symmetric matrices. The cost function $J$ represents the objectives of the attacker -- the term $\left(x_t-x^*\right)^T Q_t \left(x_t-x^*\right)$ penalizes deviation from the target state $x^*$ and the term ${\nu_t^e}^TR_t\nu_t^e $ penalizes having a large value in the detection statistic. The goal of the attacker is to design an attack sequence $e_0, \dots, e_N$ that minimizes $J$. Unlike standard LQG cost functions (e.g.,~\eqref{eqn: sysLQG}), which directly penalize the use of control, the cost function~\eqref{eqn: AttackerLQG1} does not directly penalize the attack $e_t$. Instead, $J$ only penalizes the attack indirectly through the detection statistic component  ${\nu_t^e}^TR_t\nu_t^e$. Thus, part of designing an optimal attack sequence is to determine how an attack sequence affects the CPS' Kalman filter, LQG controller, and $\chi^2$ detector.

The attacker chooses $Q_t$ and $R_t$ to match his or her own objectives. Instead of directly studying the effect of $Q_t$ and $R_t$ on the attack performance, we consider the modified cost function
\begin{equation}\label{eqn: modifiedJ}
	J_\alpha = \mathbb{E}\left[ \sum_{t = 0}^N \alpha \left(x_t - x^*\right)^T Q_t \left(x_t - x^*\right) +  {\nu_t^e}^T{R}_t \nu_t^e\right],
\end{equation}
where the attacker chooses $\alpha > 0$ as a weighting parameter. An attacker who is more concerned with achieving a low control error chooses a high value of $\alpha$, and an attacker who is more concerned with achieving a low detection statistic cost chooses a low value of $\alpha$. Define the detection statistic cost, $J_d$, and the normalized control error cost, $J_c$, as
\begin{align}
	J_{d} &= \mathbb{E}\left[ \sum_{t=0}^N \nu_t^e R_t \nu_t^e \right],\label{eqn: JDDef}\\
	J_c &= \mathbb{E}\left[\sum_{t=0}^N \left(x_t - x^*\right)^T Q_t \left( x_t - x^* \right) \right]. \label{eqn: JCDef}
\end{align}
We are interested in how the attacker's choice of $\alpha$ affects $J_d$ and $J_c$. 

\subsection{Problem Statement}\label{sect: problemStatement}
Consider the system~\eqref{eqn: ssModel1}. First, we determine the optimal cost
\begin{equation}
	J^* = \inf_{e_0, \dots, e_{N}} J,
\end{equation}
and find a sequence of attacks $e_0, \dots, e_{N}$ that achieves $J^*$. Second, we consider the modified cost function $J_\alpha$ and derive expressions for the detection statistic cost, $J_d$, and the control error cost, $J_c$, associated with the sequence of attacks $e_0, \dots, e_{N}$ that minimizes $J_\alpha$. 

\section{Attack Effect}\label{sect: attackBias}
Before finding an attack sequence that minimizes $J$, we determine how an attack sequence affects the cyber-physical system. Specifically, we calculate the effect of an attack sequence on the system's Kalman filter, its LQG controller, and the $\chi^2$ attack detector. 
Let $\epsilon_t$ be the bias in the innovation, $\nu_t^e$, induced by the attack sequence $e_0, \dots, e_t$, and let $\omega_{t+1}$ be the bias in the system's state estimate, $\widehat{x}_t^e$, induced by the attack sequence $e_0, \dots, e_t$. That is, by definition, we have
\begin{align}
	\nu_t^e &= \nu_t^0 + \epsilon_t, \label{eqn: innovationBias}\\
	\widehat{x}_t^e &= \widehat{x}_t^0 + \omega_{t+1}. \label{eqn: estimationBias}
\end{align}

When the system is under attack, the system's input is 
\begin{equation}\label{eqn: sysAttackedLQG}
	u_t^e = L\widehat{x}_t^e = L\left(\widehat{x}_t^0 + \omega_{t+1}\right).
\end{equation}
From algebraic manipulation on the output of system~\eqref{eqn: ssModel1} and equation~\eqref{eqn: sysAttackedLQG}, the sensor measurement at time $t$ under attack~is
\begin{equation}\label{eqn: sysAttackedOutput}
	y_t^e = y_t^0 + \gamma_t + \beta_t,
\end{equation}
where
	$\gamma_t = \sum_{j = 0}^{t-1} CA^{t-1-j}BL\omega_{j+1},$ and 
	$\beta_t = \Psi e_t +  \sum_{j=0}^{t-1} CA^{t-1-j}\Gamma e_j.$
The term $\beta_t$ describes how the attack sequence $e_0, \dots, e_t$ directly influences the sensor measurement. The term $\gamma_t$ describes how the attack sequence $e_0, \dots, e_t$ indirectly influences the sensor measurement through the system's controller input -- the attack induces a bias in the system's state estimate, and the feedback control causes this bias to appear in the sensor measurement.

By inspection, we can treat $\gamma_t$ and $\beta_t$ as the outputs of state-space dynamical systems. Defining the virtual state variable $\pi_t \in \mathbb{R}^n$, we see that $\gamma_t$ is the output to the dynamical system
\begin{equation}
	\mathcal{P}_1: \left\{ \begin{array}{rl} \pi_{t+1}\!\!\!\! &= A\pi_t + ABL\omega_t, \\ \gamma_t\!\!\!\! &= C\pi_t + CBL \omega_t, \end{array} \right.
\end{equation}
with $\pi_0 = 0$. 
Defining the virtual state variable $\rho_t \in \mathbb{R}^n$, we see that $\beta_t$ is the output to the dynamical system
\begin{equation}
	\mathcal{P}_2: \left\{ \begin{array}{rl} \rho_{t+1}\!\!\!\! &= A\rho_t + \Gamma e_t, \\ \beta_t\!\!\!\! &= C\rho_t + \Psi e_t, \end{array} \right.
\end{equation}
with $\rho_0 = 0$. 
Substituting equations~\eqref{eqn: sysController}, ~\eqref{eqn: estimationBias}, and ~\eqref{eqn: sysAttackedOutput} into equations~\eqref{eqn: filterEquation} and~\eqref{eqn: kalmanPrediction}, we have, after algebraic manipulation,
\begin{align}
	\widehat{x}_t^e & = \widehat{x}_t^0 + (I_n-KC)(A+BL)\omega_t + K\left(\beta_t + \gamma_t\right). \label{eqn: attackedKalmanFilter}
\end{align}
Combining equation~\eqref{eqn: estimationBias} and~\eqref{eqn: attackedKalmanFilter}, we see that $\omega_t$ is the output to the dynamical system
\begin{equation}
	\mathcal{P}_3: \left\{\!\!\!\begin{array}{rl} \omega_{t+1}\!\!\!\!\!\! &= (I_n\!\!-\!\!KC)(A+BL)\omega_t + K\left(\gamma_t + \beta_t\right), \\ \omega_t\!\!\!\!\!\! &= \omega_t. \end{array} \right.\!\!
\end{equation}
An attack sequence $e_0, \dots, e_t$ does not affect the system's Kalman filter estimate $\widehat{x}_{-1}$, which means that $\omega_0 = 0$. 

Combining equations~\eqref{eqn: nuDef},~\eqref{eqn: estimationBias},~\eqref{eqn: sysAttackedOutput}, and the definition of $\epsilon_t$, we have
\begin{equation}\label{eqn: epsilonSystemOutput}
	\epsilon_t = -C(A+BL)\omega_t + \gamma_t + \beta_t.
\end{equation}
Equation~\eqref{eqn: epsilonSystemOutput} along with $\mathcal{P}_1$, $\mathcal{P}_2$, and $\mathcal{P}_3$ describe how the attack sequence $e_0, \dots, e_t$ affects the innovation $\nu_t^e$ and the $\chi^2$ attack detector. Figure~\ref{fig: systemInterconnectDiagram} shows the block diagram of the system relating the attack $e_t$ to the innovation bias $\epsilon_t$. 
The innovation bias $\epsilon_t$ is the output of a dynamical system that takes $e_t$ as the input and consists of the interconnection between $\mathcal{P}_1$, $\mathcal{P}_2$, and $\mathcal{P}_3$. 
Define the virtual state variable
$\theta_t = \left[\begin{array}{ccc} \omega_t^T & \pi_t^T & \rho_t^T \end{array}\right]^T.$
Then, the state-space representation
\begin{equation}\label{eqn: eEpsilonSystem}
\begin{split}
	\theta_{t+1}  &= \widehat{\mathcal{A}} \theta_t + \widehat{\mathcal{B}} e_t, \\
	\epsilon_t &= \widehat{\mathcal{C}} \theta_t + \widehat{\mathcal{D}} e_t,
\end{split}
\end{equation}
where
\begin{equation}
\begin{split}
	\widehat{\mathcal{A}} &= \left[\begin{array}{ccc} (I_n-KC)A + BL & KC & KC \\ ABL & A & 0 \\ 0 & 0 & A \end{array}\right], \\
	\widehat{\mathcal{B}} &= \left[\begin{array}{ccc} \left(K\Psi\right)^T & 0^T & \Gamma^T \end{array}\right]^T, \\
	\widehat{\mathcal{C}} &= \left[\begin{array}{ccc} -CA & C & C \end{array}\right], \\
	\widehat{\mathcal{D}} &= \Psi,
\end{split}
\end{equation}
with $\theta_0 = 0$ describes the relationship between $e_t$ and $\epsilon_t$. That is the system in equation~\eqref{eqn: eEpsilonSystem} describes how an attack sequence $e_0, \dots, e_t$ affects the $\chi^2$ detector. In addition, since the attacker knows the state at time $t=0$, $\theta_0 = 0$, and the attack $e_t$ for all $t$, he or she can determine the value of $\theta_t$ for all $t$. In particular, the attacker knows the estimate bias $\omega_t$. Thus, system~\eqref{eqn: eEpsilonSystem} also describes how an attack sequence $e_0, \dots, e_t$ affects the system's Kalman filter and the system's LQG controller. 
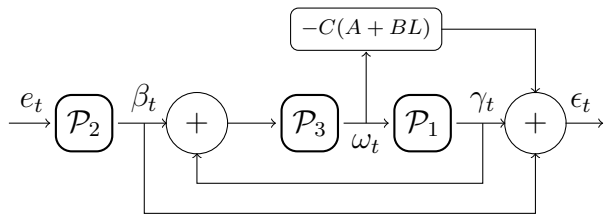
\begin{figure}[h!]
\centering
\begin{tikzpicture}
	\draw [->] (.5, 0) -- (1.05, 0);
	\node [above right] at (.5, 0){$e_t$};
	\node [rectangle, rounded corners = 6pt, draw = black, thick, minimum width = .75cm, minimum height = .75 cm] at (1.5, 0){$\mathcal{P}_2$};
	\node [circle, draw = black, thin, minimum width = .5cm] at (3, 0){$+$};
	\draw [->] (1.95, 0) -- (2.6, 0);
	\node [above] at (2.3, 0){$\beta_t$};
	\draw [->] (3.4, 0) -- (4.05, 0);
	\node [rectangle, rounded corners = 6pt, draw = black, thick, minimum width = .75cm, minimum height = .75 cm] at (4.5, 0){$\mathcal{P}_3$};
	\draw [->] (4.95, 0) -- (5.55, 0);
	\node [below] at (5.25, 0) {$\omega_t$};
	\node [rectangle, rounded corners = 6pt, draw = black, thick, minimum width = .75cm, minimum height = .75 cm] at (6, 0){$\mathcal{P}_1$};
	\draw [->] (6.45, 0) -- (7.1, 0);
	\node [above] at (6.8, 0){$\gamma_t$};
	\node [circle, draw = black, thin, minimum width = .5cm] at (7.5, 0){$+$};
	\draw [->] (7.9, 0) -- (8.4, 0);
	\node [above left] at (8.4, 0){$\epsilon_t$};
	\node [rectangle, rounded corners = 3pt, draw = black, thin, minimum height = .5cm] at (5.25, 1.25){\scriptsize$-C(A+BL)$};
	\draw [->] (5.25, 0) -- (5.25, .95);
	\draw [->] (6.25, 1.25) -- (7.5, 1.25) -- (7.5, .4);
	\draw [->] (6.8, 0) -- (6.8, -.8) -- (3, -.8) -- (3, -.4);
	\draw [->] (2.3, 0) -- (2.3, -1.2) -- (7.5, -1.2) -- (7.5, -.4);
\end{tikzpicture}
\caption{The interconnection of $\mathcal{P}_1$, $\mathcal{P}_2$, and $\mathcal{P}_3$ describes the effect of an attack on the $\chi^2$ detection statistic.}
\label{fig: systemInterconnectDiagram}
\end{figure}

\section{Optimal Attack Strategy}\label{sect: main}
In this section, we prove a sufficient condition on the matrices $Q_t$ and $R_t$ for the existence of a minimum to the cost function~$J$ and the uniqueness of the minimizing attack sequence in the time interval $t = 0, \dots, t = N-1$ (the attack $e_N$ may not be unique), and we find a cost-minimizing attack sequence $e_0, \dots, e_N$. 
Unless otherwise noted, let variables appearing without a superscript denote quantities from the attacker's Kalman filter, and let variables appearing with superscript $^e$ or $^0$ denote quantities from the system's Kalman filter. The term $\widehat{x}_t$ refers to the attacker's MMSE state estimate, defined in equation~\eqref{eqn: attackerFilterClarify} as $\widehat{x}_t  = \mathbb{E} \left[ x_t \vert \mathcal{I}_t \right]$, which is computed by the attacker's Kalman filter following equation~\eqref{eqn: attackerPredictor} and~\eqref{eqn: attackerFilter}. The term $\nu_t$ denotes the attacker's innovation at time $t$, 
\begin{equation}\label{eqn: attackerInnovationClarify}
	\nu_t = \widetilde{y}_t - C \widehat{x}_{t| t-1},
\end{equation}
where, following~\eqref{eqn: attackerPredictorClarify}, $\widehat{x}_{t|t-1} = \mathbb{E} \left[ x_{t} \vert \left\{ \mathcal{I}_{t-1}, e_{t-1} \right\} \right]$. 
\subsection{Optimal Solution}\label{sect: solution}
We find an attack sequence $e_0, \dots, e_N$ that minimizes the cost function $J$. Define
\begin{equation}
	\xi_t^T = \left[ \begin{array}{cccc} \widehat{x}_t^T & \theta_t^T & \left(\widehat{x}_t^0\right)^T & {x^*}^T \end{array} \right],\: \zeta_t = \left[\begin{array}{c} \widehat{x}_t - x^* \\ \epsilon_t + \nu_t^0 \end{array}\right].
\end{equation}
Further define
\begin{align}
	\begin{split}
		\mathcal{A} &= \left[\begin{array}{cccc} A & BL\Omega & 0 & 0 \\ 0 & \widehat{\mathcal{A}} &0 & 0 \\ 0 & 0 & A+BL & 0 \\  0 & 0 & 0 & I_n \end{array}\right], \quad
		\mathcal{B} = \left[\begin{array}{c} \Gamma + BLK\Psi \\ \widehat{\mathcal{B}} \\ 0 \\0 \end{array}\right], \quad \mathcal{K} = \left[\begin{array}{c} K \\ 0 \\ K \\ 0 \end{array}\right], \\
		\mathcal{C} &= \left[\begin{array}{c}\mathcal{H} \\ \widetilde{\mathcal{C}} \end{array}\right], \quad \mathcal{D} = \left[\begin{array}{c} 0 \\ \Psi\end{array}\right], \quad \mathcal{M} = \left[\begin{array}{c} 0 \\ I_p \end{array} \right], \quad \\
		\mathcal{H} &= \left[\begin{array}{cccc} I_n & 0 & 0 & -I_n \end{array}\right],\quad \widetilde{\mathcal{C}} = \left[\begin{array}{cccc} 0 & \widehat{\mathcal{C}} & 0 & 0 \end{array}\right],
	\end{split}\\
	\Omega &= \left[\begin{array}{ccc} (I_n-KC)A + BL & KC & KC \end{array}\right], \\
	\mathcal{F}_t &= \left[\begin{array}{cc} Q_t & 0 \\ 0 & R_t \end{array}\right].\label{eqn: FDef}
\end{align}

\begin{theorem}[Optimal Attack Strategy]\label{thm: optimalAttackSolution}
An attack sequence $e_0, \dots, e_N$ that minimizes $J$ is
\begin{align}
\begin{split}\label{eqn: LQGSolutionTerminal}
	e_N =& -\left(\mathcal{D}^T \mathcal{F}_N \mathcal{D}\right)^\dagger\left(\mathcal{D}^T \mathcal{F}_N\left(\mathcal{C} \xi_N + \mathcal{M}\nu_N\right)\right),
\end{split} \\
\begin{split}\label{eqn: LQGSolution}
	e_t =& -\left(\mathcal{D}^T \mathcal{F}_t \mathcal{D} + \mathcal{B}^T\mathcal{Q}_{t+1}\mathcal{B}\right)^{-1}\left(\left(\mathcal{D}^T \mathcal{F}_t \mathcal{C}+  \mathcal{B}^T\mathcal{Q}_{t+1}\mathcal{A}\right)\xi_t+ \mathcal{D}^T\mathcal{F}_t \mathcal{M} \nu_t\right),
\end{split}
\end{align}
for $t = 0, \dots, N-1$.
The matrix $\mathcal{Q}_t$ is given recursively backward in time by
\begin{align}
\begin{split}\label{eqn: qtDef}
	\mathcal{Q}_t =& \mathcal{C}^T\mathcal{F}_t\mathcal{C} + \mathcal{A}^T\mathcal{Q}_{t+1}\mathcal{A} - \left(\mathcal{D}^T \mathcal{F}_t \mathcal{C} + \mathcal{B}^T\mathcal{Q}_{t+1}\mathcal{A}\right)^T\left(\mathcal{D}^T \mathcal{F}_t \mathcal{D} + \mathcal{B}^T\mathcal{Q}_{t+1}\mathcal{B}\right)^{-1}\times \\
	&  \left(\mathcal{D}^T \mathcal{F}_t \mathcal{C} + \mathcal{B}^T\mathcal{Q}_{t+1}\mathcal{A}\right),
\end{split}
\end{align}
with terminal condition
\begin{align}
\begin{split}\label{eqn: QRTerminalConditions}
	\mathcal{Q}_N & = \mathcal{C}^T \mathcal{F}_N \mathcal{C} - \mathcal{C}^T \mathcal{F}_N \mathcal{D} \left(\mathcal{D}^T\mathcal{F}_N\mathcal{D}\right)^\dagger \mathcal{D}^T\mathcal{F}_N\mathcal{C}.
\end{split}
\end{align}
\end{theorem}
\begin{corollary}\label{cor: optimalSolutionExistence}
	If $Q_t$ and $R_t$ are positive definite for all $t = 0, \dots, N$, then the optimal cost $J^*$ exists, and, for $t = 0, \dots, N-1$, the optimal attack $e_t$ is unique.
\end{corollary}
\noindent The optimal attack $e_N$ may not be unique, and equation~\eqref{eqn: LQGSolutionTerminal} provides one possible attack that achieves the minimum $J^*$. 
Theorem~\ref{thm: optimalAttackSolution} states that the optimal attack strategy is a linear feedback of $\widehat{x}_t$, $\theta_t$, $\nu_t$, $x^*$, and $\widehat{x}^0_t$. Formally, we show that the variables $\widehat{x}_t$, $\theta_t$, $\nu_t$, and $\widehat{x}^0_t$ can be perfectly obtained from $\mathcal{I}_t$. 
\begin{lemma}\label{lem: eAdmissable}
	Given $\mathcal{I}_t$, the attacker can perfectly obtain the values of $\widehat{x}_t$, $\theta_t$, $\nu_t$, and $\widehat{x}_t^0$. 
\end{lemma}
\noindent The proof of Lemma~\ref{lem: eAdmissable} is in Appendix~\ref{sect: proofEAdmissable}. 

\subsection{Optimal Solution: Existence}\label{sect: existence}
Before proving Theorem~\ref{thm: optimalAttackSolution}, we provide intermediate results required to prove the sufficient condition for the existence of the optimal attack for all $t = 0, \dots, N$ and the uniqueness of the optimal attack for $t = 0, \dots, N-1$. Proofs of the intermediate results are found in the appendix.

\begin{lemma}\label{lem: QPSD}
	The matrix $\mathcal{Q}_t$ is positive semidefinite for all $t = N, N-1, \dots, 0$. Moreover, for $t = N-1, \dots, 0$, the matrix $Q_t$ follows the backward recursive relationship
\begin{equation}
	\mathcal{Q}_t  = \mathcal{X}_t^T \mathcal{F}_t \mathcal{X}_t + \mathcal{Y}_t^T \mathcal{Q}_{t+1} \mathcal{Y}_t, 
\end{equation}
where
\begin{align}
\begin{split}
	\mathcal{X}_t &= \mathcal{C} - \mathcal{D} \left( \mathcal{D}^T \mathcal{F}_t \mathcal{D} + \mathcal{B}^T \mathcal{Q}_{t+1} \mathcal{B} \right)^{-1}  \left(\mathcal{D}^T \mathcal{F}_t \mathcal{C} + \mathcal{B}^T \mathcal{Q}_{t+1} \mathcal{A} \right), 
\end{split}\\
\begin{split}
	\mathcal{Y}_t &= \mathcal{A} - \mathcal{B} \left( \mathcal{D}^T \mathcal{F}_t \mathcal{D} + \mathcal{B}^T \mathcal{Q}_{t+1} \mathcal{B} \right)^{-1} \left( \mathcal{D}^T \mathcal{F}_t \mathcal{C} + \mathcal{B}^T \mathcal{Q}_{t+1} \mathcal{A} \right),
\end{split}
\end{align}
\end{lemma}


\begin{lemma}\label{lem: convexitySufficiency}
	The matrix $\mathcal{D}^T \mathcal{F}_t \mathcal{D} + \mathcal{B}^T\mathcal{Q}_{t+1}\mathcal{B}$ is positive definite for all $t =  N-1, \dots, 0$.
\end{lemma}

\begin{lemma}\label{lem: terminalOptimumExistence}
	For all $\xi_N \in \mathbb{R}^{6n}$ and for all $\nu_N \in \mathbb{R}^p$, $\mathcal{D}^T \mathcal{F}_N \left(\mathcal{C}\xi_N + \mathcal{M}\nu_N \right) \in \mathscr{R} \left(\mathcal{D}^T \mathcal{F}_N \mathcal{D}\right)$.
\end{lemma}

\subsection{Proof of Theorem~\ref{thm: optimalAttackSolution}}\label{sect: proofs}
We first manipulate the cost function $J$. Let ${n}_t$ be the error of estimating $x_t$ from all $\mathcal{I}_t$ (i.e., $x_t = \widehat{x}_t + {n}_t$). The estimate $\widehat{x}_t$ and error $n_t$ are conditionally Gaussian and conditionally orthogonal given $\mathcal{I}_t$. From properties of conditional Gaussian processes and the Kalman Filter~\cite{Speyer}, we know that the $n_t$ are i.i.d. $\mathcal{N}\left(0, \widehat{P}\right)$, where
\[\widehat{P} = P - PC^T \Sigma^{-1}_{\nu}CP.\]  Substituting $x_t = \widehat{x}_t + {n}_t$ into the cost function $J$, we have
\begin{align}
\begin{split}
	J &= \overline{J} + \sum_{t=0}^N \trace\left(\widehat{P}Q_t\right),
\end{split}
\end{align}
where
\begin{equation}
	\overline{J} = \mathbb{E}\Bigg[\sum_{t=0}^N \bigg(\left(\widehat{x}_t-x^*\right)^T Q_t \left(\widehat{x}_t-x^*\right) + {\nu_t^e}^T R_t \nu_t^e \bigg)\Bigg]. 
\end{equation}
The attacker's Kalman filter error does not depend on the input $e_t$, so we find the optimal attack sequence by solving
\begin{equation}
	\overline{J}^* = \inf_{e_0, \dots, e_N} \overline{J}, 
\end{equation}

From the attacker's Kalman filter equations, the dynamics of $\widehat{x}_t$ follow
\begin{equation}\label{eqn: attackerKalmanDynamics}
	\begin{split}
		\widehat{x}_{t+1} =& A\widehat{x}_t + BL\left(\widehat{x}_t^0 + \Omega \theta_t + K\Psi e_t \right) + \Gamma e_t + K \nu_{t+1}, 
	\end{split}
\end{equation}
where we used that $\widehat{x}_t^e = \widehat{x}_t^0 + \omega_{t+1}$ and $\omega_{t+1} = \Omega \theta_t + K\Psi e_t$. The dynamics of $\theta_t$ are given by equation~\eqref{eqn: eEpsilonSystem}. 
To proceed, we require the following intermediate result:
\begin{lemma}\label{lem: innovationEquality}
	The attacker's innovation $\nu_t$, defined in~\eqref{eqn: attackerInnovationClarify}, satisfies $\nu_t = \nu_t^0$, where $\nu_t^0$ is the system's innovation at time $t$ in the case that there had been no attack (i.e., in the case that $e_0, \dots, e_t = 0$).
\end{lemma}
\noindent The proof of Lemma~\ref{lem: innovationEquality} is provided in Appendix~\ref{sect: proofInnovationEquality}. Then, the dynamics of $\xi_t$ are
\begin{equation}\label{eqn: xiDynamics}
	\xi_{t+1} = \mathcal{A} \xi_t + \mathcal{B} e_t + \mathcal{K} \nu_{t+1}. 
\end{equation}
 Using Lemma~\ref{lem: innovationEquality} and equation~\eqref{eqn: eEpsilonSystem}, we have
\begin{equation}\label{eqn: zetaSystemEquation}
	\zeta_t = \mathcal{C} \xi_t + \mathcal{D} e_t + \mathcal{M} \nu_t. 
\end{equation}
Substituting for $\xi_t$, $\zeta_t$, and $\zeta^*$, we rewrite $\overline{J}$ as
	$\overline{J} = \mathbb{E} \left[\sum_{t=0}^N \zeta^T \mathcal{F}_t \zeta \right].$

We now prove Theorem~\ref{thm: optimalAttackSolution}.
\begin{proof}[Proof (Theorem~\ref{thm: optimalAttackSolution})]
We resort to dynamic programming. Define $\overline{J}^*_t$ to be the optimal return function at time $t$ for information $\mathcal{I}_t$. We begin with $t = N$:
\begin{equation}\label{eqn: solutionProof1}
	\overline{J}^*_N = \inf_{e_N} \mathbb{E} \left[ \zeta_N^T \mathcal{F}_N \zeta_N \Big\vert \mathcal{I}_N \right].
\end{equation}
Substituting equation~\eqref{eqn: zetaSystemEquation} for $\zeta_N$ and setting the first derivative (with respect to $e_N$) of the right hand side of~\eqref{eqn: solutionProof1} equal to $0$, we have that the optimal $e_N$ satisfies
\begin{equation}\label{eqn: solutionProof3}
	\left(\mathcal{D}^T\mathcal{F}_N \mathcal{D} \right) e_N = \mathcal{D}^T \mathcal{F}_N\left(\mathcal{C}\xi_N + \mathcal{M}\nu_N\right). 
\end{equation}
The matrix $\mathcal{D}^T \mathcal{F}_N \mathcal{D}$ is positive semidefinite since $\mathcal{F}_N$ is positive definite. A sufficient condition for the existence of an optimal $e_N$ is $\mathcal{D}^T \mathcal{F}_N \left(\mathcal{C}\xi_N + \mathcal{M} \nu_N\right) \in \mathscr{R} \left( \mathcal{D}^T \mathcal{F}_N \mathcal{D} \right)$~\cite{Boyd}, which was shown in Lemma~\ref{lem: terminalOptimumExistence}. Then, an optimal $e_N$ is
\begin{equation}\label{eqn: solutionProof4}
	e_N =  -\left(\mathcal{D}^T \mathcal{F}_N \mathcal{D}\right)^\dagger\left(\mathcal{D}^T \mathcal{F}_N\left(\mathcal{C} \xi_N + \mathcal{M}\nu_N\right)\right). 
\end{equation}
Substituting equation~\eqref{eqn: solutionProof4} into~\eqref{eqn: solutionProof1}, we find, after algebraic manipulation, 
\begin{equation}\label{eqn: solutionProof5}
\begin{split}
	&\overline{J}^*_N= \xi_N^T \mathcal{Q}_N \xi_N + 2\xi_N^T\mathcal{R}_N {\nu_N} + {\nu_N}^T \mathcal{S}_N \nu_N + \Pi_N,
\end{split}
\end{equation}
where
\begin{align}\label{eqn: solutionProof6}
\begin{split}
	\mathcal{Q}_N &= \mathcal{C}^T \mathcal{F}_N \mathcal{C} - \mathcal{C}^T \mathcal{F}_N \mathcal{D} \left(\mathcal{D}^T\mathcal{F}_N\mathcal{D}\right)^\dagger \mathcal{D}^T\mathcal{F}_N\mathcal{C}, 
\end{split}\\
\begin{split}
	\mathcal{R}_N &= \mathcal{C}^T \mathcal{F}_N \mathcal{M} - \mathcal{C}^T \mathcal{F}_N \mathcal{D} \left(\mathcal{D}^T\mathcal{F}_N\mathcal{D}\right)^\dagger \mathcal{D}^T\mathcal{F}_N\mathcal{M}, 
\end{split} \\
\begin{split}
\mathcal{S}_N &= R_N - R_N\Psi\left(\mathcal{D}^T\mathcal{F}_N\mathcal{D}\right)^\dagger \Psi^TR_N,
\end{split}
\end{align}
and
	$\Pi_N = 0.$

Following dynamic programming, we proceed to $t = N-1$: 
\begin{equation}\label{eqn: solutionProof12}
\begin{split}
	&\overline{J}^*_{N-1} = \inf_{e_{N-1}} \mathbb{E} \left[\overline{J}_N^* + \zeta_{N-1}^T \mathcal{F}_{N-1} \zeta_{N-1} \Big| \mathcal{I}_{N-1} \right].
\end{split}
\end{equation}
Define \[\overline{\xi}_{N-1} = \mathcal{A}\xi_{N-1} + \mathcal{B} e_{N-1}.\] 
After substitution of equation~\eqref{eqn: xiDynamics} and the Kalman filter equations into~\eqref{eqn: solutionProof5} and performing algebraic manipulation, we have
\begin{equation}\label{eqn: solutionProof13}
\begin{split}
	&\mathbb{E} \left[ \overline{J}^*_N\Big| \mathcal{I}_{N-1} \right] = \overline{\xi}_{N-1}^T \mathcal{Q}_N \overline{\xi}_{N-1} + \Pi_N + \trace \left( \Sigma_\nu \left( \mathcal{K}^T \mathcal{Q}_N \mathcal{K} + 2 \mathcal{K}^T \mathcal{R}_N + \mathcal{S}_N \right)\right).
\end{split}
\end{equation}
Thus, we have
\begin{equation}\label{eqn: solutionProof15}
\begin{split}
	\overline{J}_{N-1} &= \Pi_{N-1} + \inf_{e_{N-1}}\Big\{ \overline{\xi}_{N-1}^T \mathcal{Q}_N \overline{\xi}_{N-1} +  \zeta_{N-1}^T \mathcal{F}_{N-1} \zeta_{N-1}\Big\},
\end{split}
\end{equation}
where
\begin{equation}\label{eqn: solutionProof16}
\begin{split}
	\Pi_{N-1} =& \trace \left( \Sigma_\nu \left( \mathcal{K}^T \mathcal{Q}_N \mathcal{K} + 2 \mathcal{K}^T \mathcal{R}_N + \mathcal{S}_N \right)\right) + \Pi_N.
\end{split}
\end{equation}
Then, the optimal $e_{N-1}$ is given by
\begin{equation}\label{eqn: solutionProof18}
\begin{split}
	 &e_{N-1} = -\left( \mathcal{D}^T \mathcal{F}_{N-1} \mathcal{D} + \mathcal{B}^T \mathcal{Q}_{N} \mathcal{B} \right)^{-1}\left(\left(\mathcal{D}^T \mathcal{F}_{N-1} \mathcal{C} + \mathcal{B}^T \mathcal{Q}_{N} \mathcal{A} \right) \xi_{N-1} +\mathcal{D}^T \mathcal{F}_{N-1} \mathcal{M} \nu_{N-1}\right).
\end{split}
\end{equation}
Note that a sufficient condition for the existence of a minimizing $e_{N-1}$ is the positive definiteness of \[\left( \mathcal{D}^T \mathcal{F}_{N-1} \mathcal{D} + \mathcal{B}^T \mathcal{Q}_{N} \mathcal{B} \right)\]~\cite{Boyd}, which was shown in Lemma~\ref{lem: convexitySufficiency}. 

We substitute equation~\eqref{eqn: solutionProof18} into~\eqref{eqn: solutionProof15} and perform algebraic manipulations so that $\overline{J}_{N-1}^*$ is expressed in the form
\begin{equation}\label{eqn: solutionProof19}
\begin{split}
	\overline{J}^*_{N-1}&= \xi_{N-1}^T \mathcal{Q}_{N-1} \xi_{N-1} + 2\xi_{N-1}^T\mathcal{R}_{N-1} {\nu_{N-1}} + {\nu_{N-1}}^T \mathcal{S}_{N-1} {\nu_{N-1}} + \Pi_{N-1}.
\end{split}
\end{equation}
We find the recurrence relations for the matrices in the above equation by grouping together the appropriate terms of $\overline{J}^*_{N-1}$. The matrix $\mathcal{Q}_{N-1}$ follows the recursion
\begin{equation}\label{eqn: solutionProof20}
\begin{split}
	&\mathcal{Q}_{N-1} = \mathcal{C}^T \mathcal{F}_{N-1} \mathcal{C} + \mathcal{A}^T \mathcal{Q}_{N} \mathcal{A} - \\
	& \quad \left(\mathcal{D}^T \mathcal{F}_{N-1} \mathcal{C} + \mathcal{B}^T \mathcal{Q}_N \mathcal{A} \right)^T (\mathcal{D}^T \mathcal{F}_{N-1} \mathcal{D} + \mathcal{B}^T \mathcal{Q}_N \mathcal{B})^{-1} \left(\mathcal{D}^T \mathcal{F}_{N-1} \mathcal{C} + \mathcal{B}^T \mathcal{Q}_N \mathcal{A} \right).
\end{split}
\end{equation}
The matrix $\mathcal{R}_{N-1}$ follows
\begin{equation}\label{eqn: solutionProof21}
\begin{split}
	&\mathcal{R}_{N-1} = \mathcal{C}^T \mathcal{F}_{N-1} \mathcal{M}- \left(\mathcal{D}^T \mathcal{F}_{N-1} \mathcal{C} + \mathcal{B}^T \mathcal{Q}_N \mathcal{A} \right)^T \left(\mathcal{D}^T \mathcal{F}_{N-1} \mathcal{D} + \mathcal{B}^T \mathcal{Q}_N \mathcal{B}\right)^{-1}\mathcal{D}^T \mathcal{F}_{N-1} \mathcal{M}.
\end{split}
\end{equation}
The matrix $\mathcal{S}_{N-1}$ follows\
\begin{equation}\label{eqn: solutionProof22}
\begin{split}
	\mathcal{S}_{N-1} &= R_{N-1} -  R_{N-1} \Psi\left(\mathcal{D}^T \mathcal{F}_{N-1} \mathcal{D} + \mathcal{B}^T \mathcal{Q}_N \mathcal{B}\right)^{-1} \Psi^T R_{N-1}.
\end{split}
\end{equation}

In each following iteration for $t = N-2, \dots, 0$, the optimal attack $e_t$ has the same form as equation~\eqref{eqn: solutionProof18} where we replace $N-1$ with $t$. The matrices involved in the return function $\overline{J}_t^*$ follow equations~\eqref{eqn: solutionProof16} and~\eqref{eqn: solutionProof20}--\eqref{eqn: solutionProof22} where we replace $N-1$ with~$t$. 
\end{proof}

The proof of Theorem~\ref{thm: optimalAttackSolution} states that, at each step $t$ of dynamic programming, if the matrices $Q_t$ and $R_t$ are positive definite, then there exists an optimal $e_t$. Since there exists an optimal $e_t$ for all steps $t = 0, \dots, N$, the matrices $Q_t$ and $R_t$ being positive definite is a sufficient condition for the existence of an optimal attack sequence $e_0, \dots, e_N$ and an optimal cost $J^*$. We calculate the optimal cost $J^*$ as~\cite{Speyer}
\begin{equation}
	J^* = \mathbb{E} \left[ \overline{J}^*_0 \right] + \sum_{t=0}^N \trace \left( \widehat{P} Q_t \right).
\end{equation}
We express the optimal cost in terms of the (unconditional) statistics of $\widehat{x}_{0|-1}$, the attacker's estimate of $x_0$ at time $-1$, i.e., right before beginning the attack. The quantity $\widehat{x}_{0|-1}$ is (unconditionally) distributed as $\mathcal{N} \left( \overline{x}_0, \Sigma_0 \right)$. Since $A+BL$ is  stable by definition, from the system and controller equations (\eqref{eqn: ssModel1} and~\eqref{eqn: sysController}) and the Kalman filter equations (\eqref{eqn: LimitingKalman} --~\eqref{eqn: kalmanPrediction}), we have
\begin{equation}
	\overline{x}_0 = \lim_{t \rightarrow \infty} (A+BL)^t \overline{x}_{-\infty} = 0,
\end{equation}
and $\Sigma_0$ is the unique solution to the equation
\begin{equation}
	\Sigma_0 = (A+BL)\left(\Sigma_0+ K \Sigma_{\nu} K^T\right) (A+BL)^T.
\end{equation}
Define
\begin{equation}
	\Sigma_{\xi_0}^* = \Sigma_{\xi_0} + \xi^* {\xi^*}^T,
\end{equation}
where
$\Sigma_{\xi_0} = \left[\begin{array}{cccc} I_n & 0 & I_n & 0\end{array}\right]\Sigma_0\left[\begin{array}{cccc} I_n & 0 & I_n & 0\end{array}\right]^T,$
and
${\xi^*}^T = \left[\begin{array}{cccc} 0^T & 0^T & 0^T & {x^*}^T \end{array} \right].$ Then, the cost $J^*$~is
\begin{equation}\label{eqn: JStar}
\begin{split}
	J^* = &\trace\left(\Sigma_{\xi_0}^* \mathcal{Q}_0 \right) + \sum_{t=0}^N\trace\Big(\widehat{P}Q_t + \Sigma_\nu \left( \mathcal{K}^T \mathcal{Q}_t \mathcal{K} + 2 \mathcal{K}^T \mathcal{R}_t + \mathcal{S}_t \right)\Big).
\end{split}
\end{equation}

\section{Weighted Cost Function Performance}\label{sect: weightedCostFunction}
This section studies the effect of the relative weighting between the control error component and detection statistic energy component of the cost function on the performance of the optimal attack. Consider the modified function given by~\eqref{eqn: modifiedJ}.
For any $\alpha > 0$, Theorem~\ref{thm: optimalAttackSolution} gives the sequence of attacks $e_0, \dots, e_N$ that minmizes $J_\alpha$, which has the general form
\begin{equation}\label{eqn: generalOptimalAttack}
	e_{\alpha, t} = \mathcal{L}_{\alpha, t} \xi_t + \mathcal{O}_{\alpha, t} \nu_t.
\end{equation}
We study how the attacker's choice of $\alpha$ affects the values of $J_d$ and $J_c$, and we analyze the asymptotic behavior of $J_d$ and $J_c$ as alpha approaches $0$ and $\infty$, respectively. 

\subsection{$\alpha$-Weighted Performance}
We derive expressions for the detection statistic cost and the normalized control error cost when the attack has the form of equation~\eqref{eqn: generalOptimalAttack} (i.e., when the attack is the optimal attack for the $\alpha$-weighted cost function, $J_\alpha$). Let $J_{\alpha, d}$ and $J_{\alpha, c}$ be the values of $J_d$ and $J_c$ when the attack is $e_{\alpha, t}$. 

\begin{theorem}[$\alpha$-Weighted Performance]\label{thm: alphaSweep}
	Let $e_0, \dots, e_N$ be the sequence of attacks that minimizes $J_\alpha$ given by equation~\eqref{eqn: generalOptimalAttack}. Then, the detection statistic cost, $J_{\alpha, d}$, and the normalized control error cost, $J_{\alpha, c}$, are given as
\begin{align}
\begin{split}\label{eqn: JDAlpha}
	J_{\alpha, d} = &\trace\left(\Sigma_{\xi_0}^* \widetilde{\mathcal{Q}}_{\alpha, 0} \right) + \sum_{t=0}^N\trace\Big(\Sigma_\nu \left( \mathcal{K}^T \widetilde{\mathcal{Q}}_{\alpha, t} \mathcal{K} + 2 \mathcal{K}^T \widetilde{\mathcal{R}}_{\alpha, t} + \widetilde{\mathcal{S}}_{\alpha, t} \right)\Big),
\end{split}\\
\begin{split}\label{eqn: JCAlpha}
	J_{\alpha, c} = &\trace\left(\Sigma_{\xi_0}^* \overline{\mathcal{Q}}_{\alpha, 0} \right) + \sum_{t=0}^N\trace\Big(\widehat{P}Q_t + \Sigma_\nu \left( \mathcal{K}^T \overline{\mathcal{Q}}_{\alpha, t} \mathcal{K} + 2 \mathcal{K}^T \overline{\mathcal{R}}_{\alpha, t} + \overline{\mathcal{S}}_{\alpha, t} \right)\Big). 
\end{split}
\end{align}
The matrices $\widetilde{\mathcal{Q}}_{\alpha, t}$ and $\overline{\mathcal{Q}}_{\alpha, t}$ are given by the recursive relations
\begin{align}
\begin{split}\label{eqn: tildeQRecursion}
	\widetilde{\mathcal{Q}}_{\alpha, t} =& \left(\mathcal{A} + \mathcal{B} \mathcal{L}_{\alpha, t} \right)^T\widetilde{\mathcal{Q}}_{\alpha, t+1}\left(\mathcal{A} + \mathcal{B} \mathcal{L}_{\alpha, t} \right) + \left(\widetilde{\mathcal{C}} + \Psi \mathcal{L}_{\alpha, t} \right)^T R_t \left(\widetilde{\mathcal{C}} + \Psi \mathcal{L}_{\alpha, t} \right),
\end{split}\\
\begin{split}\label{eqn: overlineQRecursion}
	\overline{\mathcal{Q}}_{\alpha, t} =& \left(\mathcal{A} + \mathcal{B} \mathcal{L}_{\alpha, t} \right)^T\overline{\mathcal{Q}}_{\alpha, t+1}\left(\mathcal{A} + \mathcal{B} \mathcal{L}_{\alpha, t} \right) + \mathcal{H}^T Q_t \mathcal{H},
\end{split}
\end{align}
with terminal conditions $\widetilde{\mathcal{Q}}_{\alpha, N+1} = \overline{\mathcal{Q}}_{\alpha, N+1} = 0$. The matrices $\widetilde{\mathcal{R}}_{\alpha, t}, \widetilde{\mathcal{S}}_{\alpha, t}, \overline{\mathcal{R}}_{\alpha, t}, \text{ and } \widetilde{\mathcal{S}}_{\alpha, t}$ are given as
\begin{align}
\begin{split}\label{eqn: tildeR}
	\widetilde{\mathcal{R}}_{\alpha, t} =& \left( \mathcal{A} + \mathcal{B} \mathcal{L}_{\alpha, t} \right)^T \widetilde{\mathcal{Q}}_{\alpha, t+1} \mathcal{B} \mathcal{O}_{\alpha, t} +  \left( \widetilde{\mathcal{C}} + \Psi \mathcal{L}_{\alpha, t} \right)^T R_t \left(I_p + \Psi \mathcal{O}_{\alpha, t}\right),
\end{split}\\
\begin{split}\label{eqn: tildeS}
	\widetilde{\mathcal{S}}_{\alpha, t} =& \mathcal{O}_{\alpha, t}^T \mathcal{B}^T \widetilde{\mathcal{Q}}_{\alpha, t+1} \mathcal{B} \mathcal{O}_{\alpha, t} + \left(I_p + \Psi \mathcal{O}_{\alpha, t}\right)^T R_t \left(I_p + \Psi \mathcal{O}_{\alpha, t}\right),
\end{split} \\
	\overline{\mathcal{R}}_{\alpha, t} =& \left( \mathcal{A} + \mathcal{B} \mathcal{L}_{\alpha, t} \right)^T \overline{\mathcal{Q}}_{\alpha, t+1} \mathcal{B} \mathcal{O}_{\alpha, t}, \label{eqn: overlineR}\\
	\overline{\mathcal{S}}_{\alpha, t} =& \mathcal{O}_{\alpha, t}^T \mathcal{B}^T \overline{\mathcal{Q}}_{\alpha, t+1} \mathcal{B} \mathcal{O}_{\alpha, t}.
\end{align}
\end{theorem}

\begin{proof}[Proof (Theorem~\ref{thm: alphaSweep})] We derive equation~\eqref{eqn: JCAlpha}, the expression for $J_{\alpha, c}$. The derivation of the expression for $J_{\alpha, d}$ follows similarly. Substituting $x_t = \widehat{x}_t + n_t$, as in the proof of Theorem~\ref{thm: optimalAttackSolution}, we have
\begin{equation}
	J_{\alpha, c} = \sum_{t=0}^N \trace\left(\widehat{P}Q_t\right) + \mathbb{E} \left[ \sum_{t=0}^N \xi_t^T \mathcal{H}^T Q_t \mathcal{H} \xi_t \right].
\end{equation}
We evaluate $\overline{J}_{\alpha, c} =  \mathbb{E} \left[ \sum_{t=0}^N \xi_t^T \mathcal{H}^T Q_t \mathcal{H} \xi_t \right]$ recursively backward in time. Define $\overline{J}_{\alpha, c, t}$ to be the cost-to-go function for $\overline{J}_{\alpha, c}$ when the attack is $e_{\alpha, t}$ and the attacker's information is~$\mathcal{I}_t$. 

We begin at time $t=N$, for which we have
\begin{align}
	\overline{J}_{\alpha, c, N} &= \mathbb{E} \left[ \xi_N^T \mathcal{H}^T Q_N \mathcal{H} \xi_N | \mathcal{I}_N \right], \\
		&= \xi_N^T \mathcal{H}^T Q_N \mathcal{H} \xi_N.
\end{align}
Note that, at time $t=N$, the attack $e_N$ does not affect the cost-to-go function, since $\xi_N$ does not depend on $e_N$. Rearranging $\overline{J}_{\alpha, c, N}$, we have
\begin{equation}
\begin{split}
	\overline{J}_{\alpha, c, N} =& \xi_N^T \overline{\mathcal{Q}}_{\alpha, N} \xi_N + 2 \xi_N^T \overline{\mathcal{R}}_{\alpha, N} \nu_N  + \nu_N^T \overline{\mathcal{S}}_{\alpha, N} \nu_N +\overline{\Pi}_{N},
\end{split}
\end{equation}
where $\overline{\mathcal{Q}}_{\alpha, N} = \mathcal{H}^T Q_N \mathcal{H}$, $\overline{\mathcal{R}}_{\alpha, N} = 0$ ($\overline{\mathcal{R}}_{\alpha, N} \in \mathbb{R}^{6n \times p}$), $\overline{\mathcal{S}}_{\alpha, N} = 0$ ($\overline{\mathcal{S}}_{\alpha, N} \in \mathbb{R}^{p \times p}$), and $\overline{\Pi}_N = 0$.

Proceeding to $t = N-1$, we have, following~\eqref{eqn: generalOptimalAttack}, that the attack is \[e_{\alpha, N-1} = \mathcal{L}_{\alpha, N-1} \xi_{N-1} + \mathcal{O}_{\alpha, N-1} \nu_{N-1}.\] The cost-to-go function at $t=N-1$ is
\begin{align}
\begin{split}\label{eqn: JCAlpha1}
	\overline{J}_{\alpha, c, N-1} &= \mathbb{E} \Big[ \xi_{N-1}^T \mathcal{H}^T Q_{N-1}\mathcal{H} \xi_{N-1} + \overline{J}_{\alpha, c, N} \Big| \mathcal{I}_{N-1} \Big].
\end{split}
\end{align}
Substituting $\overline{\xi}_{\alpha, N-1} = \mathcal{A}\xi_{N-1} + \mathcal{B}e_{\alpha, N-1}$ into~\eqref{eqn: JCAlpha1} and evaluating the conditional expectation, we have
\begin{equation}\label{eqn: JCAlpha2}
\begin{split}
	\overline{J}_{\alpha, c, N-1} &=  \xi_{N-1}^T \mathcal{H}^T Q_{N-1}\mathcal{H} \xi_{N-1} + \overline{\Pi}_{N-1} +  \overline{\xi}_{N-1}^T \overline{\mathcal{Q}}_{\alpha, N} \overline{\xi}_{N-1},
\end{split}
\end{equation}
where
\begin{equation}\label{eqn: overlinePi}
\begin{split}
	\overline{\Pi}_{N-1} =&\overline{\Pi}_N + \trace \Big( \Sigma_\nu \left( \mathcal{K}^T \overline{\mathcal{Q}}_{\alpha, N} \mathcal{K} + 2 \mathcal{K}^T \overline{\mathcal{R}}_{\alpha, N} + \overline{\mathcal{S}}_{\alpha, N} \right)\Big).
\end{split}
\end{equation}
After further substituting for $e_{\alpha, N-1}$ into~\eqref{eqn: JCAlpha2} and performing algebraic manipulations, we have
\begin{equation}\label{eqn: JCAlpha3}
\begin{split}
	\overline{J}_{\alpha, c, N-1} & =\overline{\Pi}_{N-1} + \xi_{N-1}^T \overline{\mathcal{Q}}_{\alpha, N-1} \xi_{N-1} + 2 \xi_{N-1}^T \overline{\mathcal{R}}_{\alpha, N-1} \nu_{N-1} +  \nu_{N-1}^T \overline{\mathcal{S}}_{\alpha, N-1} \nu_{N-1},
\end{split}
\end{equation}
where
\begin{align}
\begin{split}\label{eqn: overlineQRecursionN}
	\overline{\mathcal{Q}}_{\alpha, N-1} =& \left(\mathcal{A} + \mathcal{B} \mathcal{L}_{\alpha, N-1} \right)^T\overline{\mathcal{Q}}_{\alpha, N}\left(\mathcal{A} + \mathcal{B} \mathcal{L}_{\alpha, N-1} \right) + \mathcal{H}^T Q_{N-1} \mathcal{H},
\end{split}\\
\overline{\mathcal{R}}_{\alpha, N-1} =& \left( \mathcal{A} + \mathcal{B} \mathcal{L}_{\alpha, N-1} \right)^T \overline{\mathcal{Q}}_{\alpha, N} \mathcal{B} \mathcal{O}_{\alpha, N-1}, \label{eqn: overlineRN}\\
	\overline{\mathcal{S}}_{\alpha, N-1} =& \mathcal{O}_{\alpha, N-1}^T \mathcal{B}^T \overline{\mathcal{Q}}_{\alpha, N} \mathcal{B} \mathcal{O}_{\alpha, N-1}. \label{eqn: overlineSN}
\end{align}

To complete the proof, we iterate through all $t = N-2, \dots, 0$ and compute $\overline{J}_{\alpha, c, t}$. We find the same relations as described by  equations~\eqref{eqn: overlinePi},~\eqref{eqn: overlineQRecursionN}-\eqref{eqn: overlineSN}, with $t$ replacing $N-1$. The control error cost is given as \[J_{\alpha, c} = \mathbb{E} \left[ \overline{J}_{\alpha, c, 0}\right] + \sum_{t=0}^N \trace \left(\widehat{P} Q_t \right).\]
Computing the expectation in the above expression and expressing the results in terms of the statistics of $\widehat{x}_{0|-1}$ (similar to the procedure for calculating the optimal cost in Section~\ref{sect: proofs}), we derive equation~\eqref{eqn: JCAlpha}.
\end{proof}

Theorem~\ref{thm: alphaSweep} provides explicit expressions for the normalized control error cost and detection statistic cost incurred by an attack sequence\footnote{Equations~\eqref{eqn: JDAlpha} and~\eqref{eqn: JCAlpha} are valid for all attack sequences of the form $e_{\alpha, t} = \mathcal{L}_{\alpha, t} \xi_{t} + \mathcal{O}_{\alpha, t}\nu_t$ (i.e., the matrices $\mathcal{L}_{\alpha, t}$ and $\mathcal{O}_{\alpha, t}$ do not have to come from the optimal attack expressions in Theorem~\ref{thm: optimalAttackSolution}).}. Using Theorem~\ref{thm: alphaSweep}, the attacker can compute $J_{\alpha, d}$ and $J_{\alpha, c}$, the attacker can determine how the parameter $\alpha$, the relative weighting between the control error component and the detection statistic component of the cost function, affects the performance of the attack.  In Section~\ref{sect: example}, we provide a numerical example in which we calculate $J_{\alpha, d}$ and $J_{\alpha, c}$ as a function of $\alpha$. 

\subsection{Asymptotic Weighting Performance}
The previous subsection derives expressions for $J_{\alpha, d}$ and $J_{\alpha, c}$ for arbitrary values of $\alpha > 0$. In this section, we consider the detection statistic cost and control error cost incurred by the attack sequence minimizing $J_\alpha$ as $\alpha$ approaches $0$ and as $\alpha$ approaches $\infty$. Let
\begin{align}
	J_{0, d} = \lim_{\alpha \rightarrow 0^+} J_{\alpha, d}\:, \quad J_{\infty, d} = \lim_{\alpha \rightarrow \infty} J_{\alpha, d},
\end{align}
and let $J_{0, c}$ and $J_{\infty, c}$ be defined similarly. In order to compute $J_{0, d}$, $J_{\infty, d}$, $J_{0, c}$, and $J_{\infty, c}$, we need to compute $e_{0, t}$ and $e_{\infty, t}$. The attack (sequence) $e_{0, t}$ is the attack sequence $e_0, \dots, e_N$ that minimizes $J_d$, and the attack (sequence) $e_{\infty, t}$ is the attack sequence $e_0, \dots, e_N$ that minimizes $J_c$. From a practical perspective, an attacker who chooses a weighting parameter $\alpha$ vanishing to $0$ is one who wishes to minimize the sum of the detection statistic (i.e., an attacker who only wishes to avoid detection), and an attacker who chooses a weighting parameter $\alpha$ approaching $\infty$ is an attacker who wishes to minimize the sum of the control error (i.e., an attacker who only wishes to achieve the control objective and does not care about avoiding detection).

Solving for the attack sequence that minimizes $J_d$ is equivalent to solving for the attack sequence that minimizes $J$ (equation~\eqref{eqn: AttackerLQG1}) with $Q_t = 0$. In this case, we cannot use Theorem~\ref{thm: optimalAttackSolution} directly to solve for the optimal attack sequence since Theorem~\ref{thm: optimalAttackSolution} presupposes that $Q_t, R_t \succ 0$. The same restriction applies to finding the attack sequence that minimizes $J_c$. Instead, we separately solve for the optimizing attack sequence when the weighting parameter $\alpha$ is equal to either $0$ or $\infty$. In equations~\eqref{eqn: JStar},~\eqref{eqn: JDAlpha}, and~\eqref{eqn: JCAlpha}, the term $\Sigma_{\xi_0}^*$ captures the dependence of the optimal cost, detection cost, and normalized control cost, respectively, on the target state $x^*$. Intuitively, a target that is farther away results in a higher optimal cost. That is, if, there are two separate target states $x^*_1$ and $x^*_2$, and, $x^*_2 = cx^*_1$ for some $c > 1$, then the cost associated with target state $x^*_2$ is at least the cost associated with the target state $x^*_1$.

\begin{theorem}[Asymptotically Weighted Optimal Attacks]\label{thm: asymptoticAttack}
	For $\alpha = 0$ and $\alpha = \infty$, the optimal attack strategy has the form
	\begin{equation}\label{eqn: asymptoticAttack}
		e_{\alpha, t} = \mathcal{L}_{\alpha, t} \xi_t + \mathcal{O}_{\alpha, t} \nu_t.
	\end{equation}

	For $\alpha = 0$, and for $t=0, \dots, N$, the matrices $\mathcal{L}_{0, t}$ and $\mathcal{O}_{0, t}$ follow
	\begin{align}
		\begin{split}
			\mathcal{L}_{0, t} =& -\left(\Psi^T R_t \Psi + \mathcal{B}^T \mathcal{V}_{t+1} \mathcal{B} \right)^\dagger \left(\Psi^T R_t \widetilde{\mathcal{C}} + \mathcal{B}^T\mathcal{V}_{t+1} \mathcal{A} \right),
		\end{split}\\
		\mathcal{O}_{0, t} =& -\left(\Psi^T R_t \Psi + \mathcal{B}^T \mathcal{V}_{t+1} \mathcal{B} \right)^\dagger \Psi^T R_t,
	\end{align}
	where $\mathcal{V}_t$ follows the recursion
	\begin{equation}
		\begin{split}
			\mathcal{V}_t =& \widetilde{\mathcal{C}}^T R_t \widetilde{\mathcal{C}} + \mathcal{A}^T \mathcal{V}_{t+1} \mathcal{A} -  \left(\Psi^T R_t \widetilde{\mathcal{C}} + \mathcal{B}^T\mathcal{V}_{t+1} \mathcal{A} \right)^T \left(\Psi^T R_t \Psi + \mathcal{B}^T \mathcal{V}_{t+1} \mathcal{B} \right)^\dagger \times \\ &\left(\Psi^T R_t \widetilde{\mathcal{C}} + \mathcal{B}^T\mathcal{V}_{t+1} \mathcal{A} \right),
		\end{split}
	\end{equation}
	with $\mathcal{V}_{N+1} = 0$. 

	For $\alpha = \infty$, and for $t = 0, \dots, N$, the matrices $\mathcal{L}_{\infty, t}$ and $\mathcal{O}_{0, t}$ follow
	\begin{align}
		\begin{split}
			\mathcal{L}_{\infty, t} =& -\left(\mathcal{B}^T \mathcal{W}_{t+1} \mathcal{B} \right)^\dagger \mathcal{B}^T \mathcal{W}_{t+1} \mathcal{A},
		\end{split}\\
		\mathcal{O}_{\infty, t} =& 0,
	\end{align}
	where $\mathcal{W}_t$ follows the recursion
	\begin{equation}
		\begin{split}\label{eqn: WRecursion}
			\mathcal{W}_t =& \mathcal{H}^T Q_t \mathcal{H} + \mathcal{A}^T \mathcal{W}_{t+1} \mathcal{A} - \mathcal{A}^T \mathcal{W}_{t+1} \mathcal{B} \left(\mathcal{B}^T \mathcal{W}_{t+1} \mathcal{B} \right)^\dagger \mathcal{B}^T \mathcal{W}_{t+1} \mathcal{A},
		\end{split}
	\end{equation}
	with $\mathcal{W}_{N+1} = 0$. 
\end{theorem}
\noindent The main difference between Theorem~\ref{thm: optimalAttackSolution} and Theorem~\ref{thm: asymptoticAttack} is that, when the weighting parameter $\alpha$ is either $0$ or $\infty$, the optimal attack sequence over the time interval $t = 0, \dots, t = N -1$ may not be unique. Before we prove Theorem~\ref{thm: asymptoticAttack}, we require the following intermediate result, the proof of which is found in the appendix. 
\begin{lemma}\label{lem: asymptoticIntermediate}
	For all $\xi_t \in \mathbb{R}^{6n}$, $\mathcal{B}^T \mathcal{W}_{t+1} \mathcal{A} \xi_t \in \mathscr{R} \left( \mathcal{B}^T \mathcal{W}_{t+1} \mathcal{B} \right)$.
\end{lemma}
\noindent We now provide the proof for Theorem~\ref{thm: asymptoticAttack} for the case of $\alpha = \infty$; the proof for the case of $\alpha = 0$ follows similarly.

\begin{proof}[Proof (Theorem~\ref{thm: asymptoticAttack})]
	We consider the case of $\alpha = \infty$ and find the sequence of attacks that solves
	\begin{equation}\label{eqn: asymptoticProof1}
		J_{\infty, c}^* = \inf_{e_0, \dots, e_N} \mathbb{E} \left[\sum_{t=0}^N \left(x_t - x^*\right)^T Q_t \left(x_t - x^* \right) \right].
	\end{equation}
	Substituting $x_t = \widehat{x}_t + n_t$ into equation~\eqref{eqn: asymptoticProof1}, we have that an attack sequence that minimizes the cost function~\eqref{eqn: asymptoticProof1} is an attack sequence that minimizes $\overline{J}_{\infty, c} =  \mathbb{E} \left[ \sum_{t=0}^N \xi_t^T \mathcal{H}^T Q_t \mathcal{H} \xi_t \right]$. We resort to dynamic programming to find the minimizing attack sequence.

Define $\overline{J}^*_{\infty, c, t}$ to be the optimal return function at time $t$ for information $\mathcal{I}_t$. At time $t=N$ we have
\begin{align}
	\overline{J}^*_{\infty, c, N} &= \mathbb{E} \left[\xi_N^T \mathcal{H}^T \mathcal{Q}_N \mathcal{H} \xi_N \Big\vert \mathcal{I}_N \right],\\
		&= \xi_N^T \mathcal{W}_N \xi_N + \breve{\Pi}_N,
\end{align}
where $\mathcal{W}_N = \mathcal{H}^T \mathcal{Q}_N \mathcal{H}$, and $\breve{\Pi}_N = 0$. Note that, at time $t=N$, the attack $e_N$ does not affect the optimal return function, so we let $e_N = 0$. 

Proceeding to $t={N-1}$, we have
\begin{align}
	\begin{split}\label{eqn: asymptoticProof2}
		\overline{J}^*_{\infty, c, N-1} &= \inf_{e_{N-1}} \mathbb{E} \Big[\xi_{N-1}^T \mathcal{H}^T \mathcal{Q}_{N-1} \mathcal{H} \xi_{N-1} +  \overline{J}^*_{\infty, c, N} \Big \vert \mathcal{I}_{N-1} \Big].
	\end{split}\\
	\begin{split}\label{eqn: asymptoticProof3}
		&= \breve{\Pi}_{N-1} + \xi_{N-1}^T \mathcal{H}^T \mathcal{Q}_{N-1} \mathcal{H} \xi_{N-1}  + \\ &\qquad\inf_{e_{N-1}} \left(\mathcal{A}\xi_{N-1} + \mathcal{B}e_{N-1}\right)^T \mathcal{W}_N \left(\mathcal{A}\xi_{N-1} + \mathcal{B}e_{N-1}\right),
	\end{split}
\end{align}
where $\breve{\Pi}_{N-1} = \breve{\Pi}_{N} + \trace\left(\Sigma_\nu \mathcal{K}^T \mathcal{W}_N \mathcal{K} \right)$. From equation~\eqref{eqn: asymptoticProof3}, we have that $e_{\infty, N-1}$ is given by
\begin{equation}\label{eqn: asymptoticProof4}
	e_{\infty, N-1} = -\left(\mathcal{B}^T \mathcal{W}_N \mathcal{B} \right)^\dagger \mathcal{B}^T \mathcal{W}_N \mathcal{A} \xi_{N-1}.
\end{equation}
A sufficient condition for the existence of an optimal $e_{\infty, N-1}$ is $\mathcal{B}^T \mathcal{W}_N \mathcal{A} \xi_{N-1} \in \mathscr{R} \left( \mathcal{B}^T \mathcal{W}_N \mathcal{B} \right)$~\cite{Boyd}, which is satisfied according to Lemma~\ref{lem: asymptoticIntermediate}.
The optimal return function is
\begin{equation}\label{eqn: asymptoticProof5}
	\overline{J}^*_{\infty, c, N-1} = \xi_{N-1}^T \mathcal{W}_{N-1} \xi_{N-1} + \breve{\Pi}_{N-1},
\end{equation}
where
\begin{equation}\label{eqn: asymptoticProof6}
\begin{split}
	\mathcal{W}_{N-1} =& \mathcal{H}^T Q_{N-1} \mathcal{H} + \mathcal{A}^T \mathcal{W}_N \mathcal{A} - \mathcal{A}^T \mathcal{W}_N \mathcal{B} \left( \mathcal{B}^T \mathcal{W}_N \mathcal{B} \right)^\dagger \mathcal{B}^T \mathcal{W}_N \mathcal{A}.
\end{split}
\end{equation}
To complete the proof, we iterate through all $t = N-2, \dots, 0$ and compute $e_{\infty, t}$. We find the same relations as described by equations~\eqref{eqn: asymptoticProof4}-\eqref{eqn: asymptoticProof6}, with $t$ replacing $N-1$. This completes the proof for the case of $\alpha = \infty$, and the case of $\alpha = 0$ follows similarly.
\end{proof}

Theorem~\ref{thm: asymptoticAttack} provides expressions for the optimal attack sequence when the weighting parameter $\alpha$ is either $0$ or $\infty$. For both of these choices of $\alpha$, the optimal attack has the form given in equation~\eqref{eqn: generalOptimalAttack}. Thus, combining the results of Theorem~\ref{thm: asymptoticAttack} with the results of Theorem~\ref{thm: alphaSweep}, we find expressions for $J_{0, d}, J_{\infty, d}, J_{0, c},$ and $J_{\infty, c}$. The quantities $J_{0, d}$ and $J_{\infty, c}$ are lower bounds on $J_{d}$ and $J_{c}$ respectively. The quantity $J_{\infty, d}$ is the detection cost associated with an attack strategy that only seeks to minimize the control error cost. The quantity $J_{0, c}$ is the control error cost associated with an attack strategy that only seeks to minimize the detection cost. In Section~\ref{sect: example}, we provide a numerical example in which we calculate these quantities and compare these quantities against $J_{\alpha, d}$ and $J_{\alpha, c}$ for arbitrary values of $\alpha$. 

\section{Numerical Example}\label{sect: example}
We demonstrate the attack strategy using a numerical example of a remotely controlled helicopter under attack. Reference~\cite{Helicopter} provides a state space model of the helicopter's dynamics linearized about an equilibrium point. The dynamics are described in $10$ state variables: forward velocity ($x_{(1)}$), vertical velocity ($x_{(2)}$), pitch rate ($x_{(3)}$), pitch ($x_{(4)}$), longitudinal flapping ($x_{(5)}$), lateral velocity ($x_{(6)}$), roll rate ($x_{(7)}$), yaw rate ($x_{(8)}$), roll angle ($x_{(9)}$), and lateral flapping ($x_{(10)}$). The helicopter has four actuator inputs: the collective ($u_{(1)}$), the longitudinal cyclic ($u_{(2)}$), the tail rotor pedal ($u_{(3)}$), and the lateral cyclic ($u_{(4)}$). Due to space constraints, we do not provide the numerical values of the $A$ and $B$ matrices of the model here and instead refer the reader to~\cite{Helicopter}. In our numerical example, the system has a sensor measuring each of the $10$ state variables, so $C = I_{10}$. We use the following statistical properties in the numerical simulation: $\overline{x}_{t_0} = 0$, $\Sigma_x = 5I_{10}$, $\Sigma_v = 10^{-3} I_{10}$ and $\Sigma_w = 10^{-4} \diag (6, .1, 2, 2, .1, 6, 2, 2, 2, .1),$ where $\diag (a_1, \dots, a_q)$ is a $q$ by $q$ matrix with diagonal elements $a_1, \dots, a_q$. 

The attacker can attack the tail rotor input ($u_{(3)}$) and the lateral cyclic input ($u_{(4)}$) and can change the values of the sensors measuring vertical velocity, lateral velocity, roll rate and lateral flapping. The attacker's target state is
\[ x^* = [\begin{array}{cccccccccc} 0 & 4 & 0 & 0 & 0 & 8.2 & 0 & 0 & 0 & 0 \end{array}]^T.\]
We consider an attacker that only wishes to manipulate the vertical velocity ($x_{(2)}$) and lateral velocity ($x_{(6)}$), so, the attacker's $Q_t$ matrix is given as
\[Q_t = \diag (\begin{array}{cccccccccc} .1,& 3,& .1,& .1,& .1, &4, &.1, &.1, &.1, &.1\end{array}). \]
For the matrix $R_t$, we have  $R_t = \Sigma_\nu ^{-1},$ which can be calculated from the statistical properties of the system using the Kalman filter equations. 

The system starts running at $t = -200$, and the attack occurs over the time period $t = 0$ to $t = 200$. The effect of the attack on the vertical and lateral velocities of the helicopter is shown in Figure~\ref{fig: numericalResults1}, and the effect of the attack on the detection statistic $g_t$ is shown in Figure~\ref{fig: numericalResults2}.
\begin{figure}[h!]
	\centering
	\includegraphics[keepaspectratio = true, scale = .6]{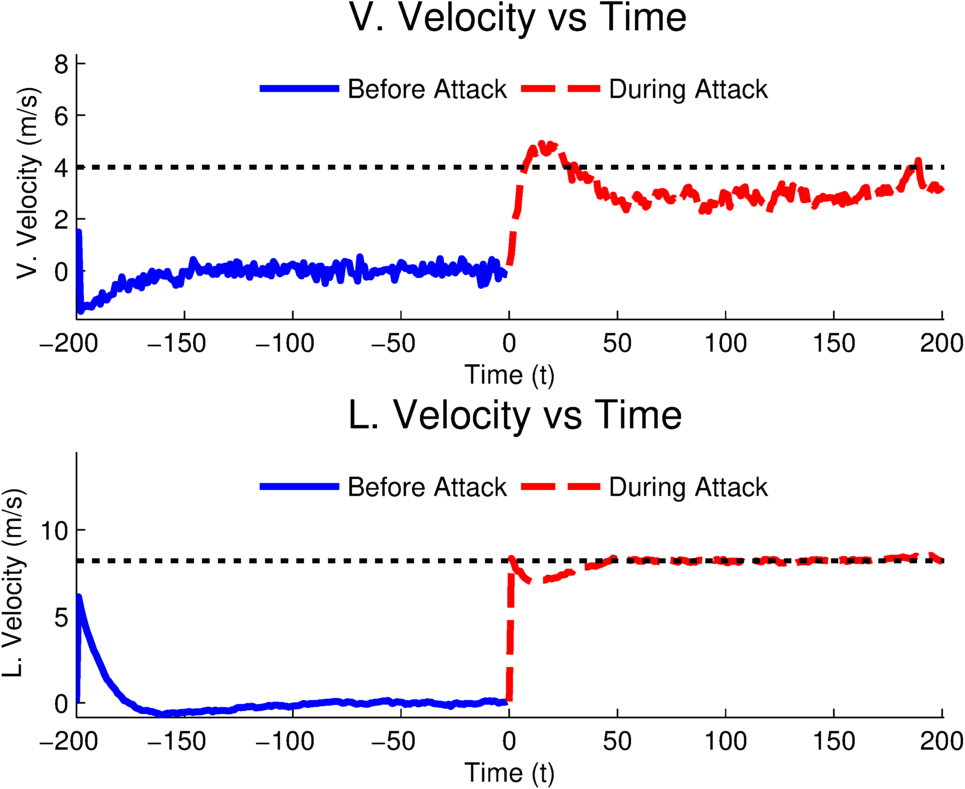}
	\caption{Effect of the attack $e_t$ on the system states. The black dotted line denotes the target state $x^*$. Top: state $x_{(2)}$ (vertical velocity) versus time. Bottom: state $x_{(6)}$ (lateral velocity) versus time.}
	\label{fig: numericalResults1}
\end{figure}
\begin{figure}[h!]
	\centering
	\includegraphics[keepaspectratio = true, scale = .6]{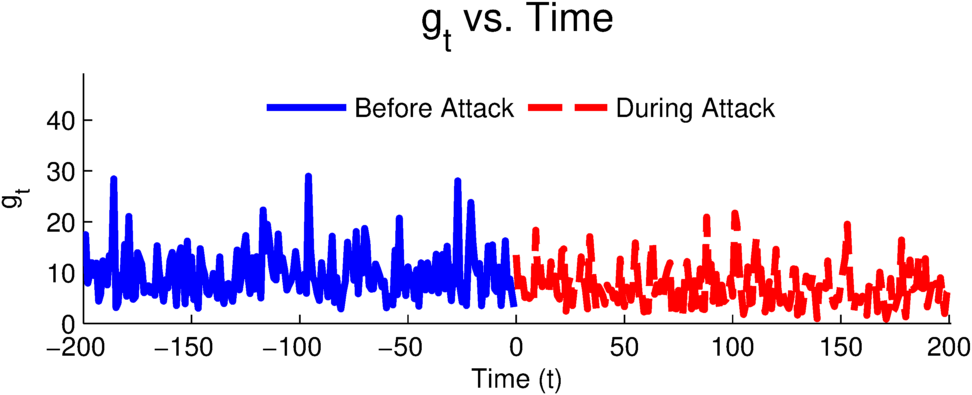}
	\caption{Effect of the attack $e_t$ on the detection statistic $g_t$.}
	\label{fig: numericalResults2}
\end{figure}
As Figure~\ref{fig: numericalResults1} shows, from time $t = -200$ to $t = -1$, the system's LQG controller maintains the equilibrium point for the helicopter. Then, from time $t=0$ to $t=-1$, the attacker successfully uses the strategy given by Theorem~\ref{thm: optimalAttackSolution} (with $\alpha = 1$) to move the system to the desired target state. The attack is effective at accomplishing the control objective even though the system's LQG controller attempts to bring the helicopter back to the equilibrium point. Figure~\ref{fig: numericalResults2} shows that the optimal attack does not cause a noticeable increase in the detection statistic. 

By choosing the relative weighting between the control error component and the detection statistic energy component of the cost function, the attacker changes the performance of the optimal attack. We consider the remotely controlled helicopter under attack when the attacker places weighting parameter $\alpha > 0$ on the control error component of the cost function. First, using Theorem~\ref{thm: optimalAttackSolution}, we compute the optimal attack sequence for values of $\alpha$ ranging from $10^{-6}$ to $10^{6}$. Then, using Theorem~\ref{thm: alphaSweep}, we compute the detection statistic cost $J_d$ and the normalized control error cost $J_c$ for the optimal attack sequence associated with each value of $\alpha$. From Theorem~\ref{thm: asymptoticAttack}, we compute the values of $J_{0, d}, J_{\infty, d}, J_{0, c},$ and $J_{\infty, c}$ for the helicopter under attack.

\begin{figure}[h!]
	\centering
	\includegraphics[keepaspectratio = true, scale = .6]{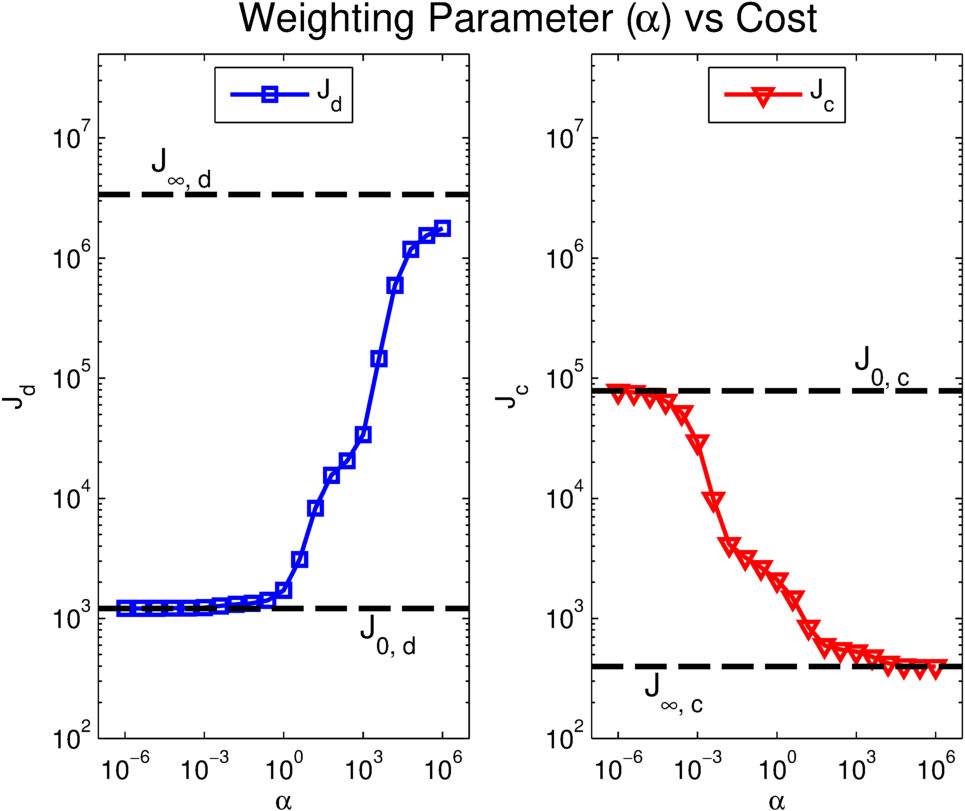}
	\caption{Effect of weighting parameter $\alpha$ on $J_d$ (left) and $J_c$ (right).}
	\label{fig: alphaSweep}
\end{figure}

 Figure~\ref{fig: alphaSweep} shows that, as the attacker increases the weighting on the control error component, the optimal attack has lower control error cost but higher detection cost. As $\alpha$ approaches $0$, the two cost components $J_{\alpha, d}$ and $J_{\alpha, c}$ approach the bounds $J_{0, d}$ and $J_{0, c}$ calculated using Theorem~\ref{thm: alphaSweep} and Theorem~\ref{thm: asymptoticAttack}. Similarly, as $\alpha$ approaches $\infty$, the two cost components $J_{\alpha, d}$ and $J_{\alpha, c}$ approach $J_{\infty, d}$ and $J_{\infty, c}$. From a practical perspective, Theorem~\ref{thm: alphaSweep} tells an attacker how the weighting parameter $\alpha$ affects the detection statistic and the control error. This allows the attacker to choose a value of $\alpha$ that best suits his or her goals.

\section{Conclusion}\label{sect: conclusion}
In this paper we have studied an attack strategy against a cyber-physical system in which an attacker wishes to move the system to a target state while avoiding the system's $\chi^2$ attack detector. We have formulated a quadratic cost function that captures the attacker's objectives; the cost function is the sum of components penalizing the distance from the target state and the energy of the system's detection statistic, respectively. Using dynamic programming, we have found an optimal sequence of attacks, which is a linear feedback of the attacker's state estimate, that minimizes this cost function. Moreover, we have derived expressions for the control error component and the detection statistic component of the cost as a function of a cost weighting parameter chosen by the attacker. Finally, we have provided a numerical example that demonstrates the attack strategy and have shown a trade-off between control performance and detection avoidance. Future work includes determining optimal attack strategies for a less powerful attacker, e.g., an attacker who does not exactly know the system model or an attacker who does not exactly know the sensor output. 
\appendices
\section{Proof of Lemma~\ref{lem: eAdmissable}}\label{sect: proofEAdmissable}
\begin{proof}
By definition, (equation~\eqref{eqn: attackerFilterClarify}), we have $\widehat{x}_t = \mathbb{E} \left[ x_t \vert \mathcal{I}_t \right]$, so $\widehat{x}_t$ can be perfectly obtained from $\mathcal{I}_t$. Similarly, from equation~\eqref{eqn: attackerInnovationClarify}, we have $ \nu_t = \widetilde{y}_t - C \mathbb{E} \left[ x_{t-1} \vert \left\{\mathcal{I}_{t-1}, e_{t-1} \right\} \right]$, which means that $\nu_t$ can be perfectly obtained from $\left\{\mathcal{I}_{t-1}, e_{t-1}, \widetilde{y}_t \right\} = \mathcal{I}_t$.

From the dynamical system in equation~\eqref{eqn: eEpsilonSystem}, we know that
\begin{equation}
	\theta_t = \widehat{\mathcal{A}}^t \theta_0 + \sum_{j = 0}^{t-1} \widehat{\mathcal{A}}^{t-1-j}\widehat{\mathcal{B}}e_j,
\end{equation}
where $\theta_0 = 0$. From equation~\eqref{eqn: attackerInformation} we have
\[ \left\{e_0, e_1, \dots, e_{t-1} \right\} \subset \mathcal{I}_t, \]
so the information set $\mathcal{I}_t$ exactly determines $\theta_t$. 

Since one can run a Kalman filter, one can also run a Kalman filter while ignoring the input $e_t$ (i.e., perform Kalman filtering treating all inputs as $e_t = 0$). This process computes $\widehat{x}_t^e$, the estimate produced by the system's Kalman filter. From the Kalman filtering equations, we know that
\begin{equation}\label{eqn: eAdmissableProof1}
	\widehat{x}_t^e = \left(A+BL\right)\widehat{x}_{t-1}^e + K \left(y_t^e - C(A+BL) \widehat{x}_{t-1}^e \right). 
\end{equation}
Substituting equations~\eqref{eqn: estimationBias} and~\eqref{eqn: eEpsilonSystem} into the left hand side of~\eqref{eqn: eAdmissableProof1} and performing algebraic manipulation, we have
\begin{equation}\label{eqn: eAdmissableProof2}
	\widehat{x}_t^0 = (A+BL) \widehat {x}_{t-1}^e - \Omega \theta_t + \left(\widetilde{y}_t - C(A+BL) \widehat{x}_{t-1}^e \right),
\end{equation}
where $\widehat{x}_{t-1}^e$ is exactly determined given $\widetilde{y}_0, \dots, \widetilde{y}_{t-1}$ and $e_0, \dots, e_{t-1}$, which are contained in $\mathcal{I}_t$. In addition, we have $\widetilde{y}_{t} \in \mathcal{I}_t$. By equation~\eqref{eqn: eAdmissableProof2}, the information set $\mathcal{I}_t$ exactly determines $\widehat{x}_{t}^0$. 
\end{proof}

\section{Proof of Lemma~\ref{lem: QPSD}}
\begin{proof}
We resort to induction.

\underline{Base Case}: We show that $\mathcal{Q}_N \succeq 0$. By algebraic manipulation on equation~\eqref{eqn: QRTerminalConditions}, we have
\begin{align}
\begin{split}\label{eqn: QNPSD}
	\mathcal{Q}_N & = \Big(\mathcal{C} - \mathcal{D}\left(\mathcal{D}^T\mathcal{F}_N\mathcal{D}\right)^\dagger\left(\mathcal{D}^T\mathcal{F}_N\mathcal{C}\right)\Big)^T\mathcal{F}_N \Big(\mathcal{C} - \mathcal{D}\left(\mathcal{D}^T\mathcal{F}_N\mathcal{D}\right)^\dagger\left(\mathcal{D}^T\mathcal{F}_N\mathcal{C}\right)\Big)
\end{split}
\end{align}
By definition, $\mathcal{F}_N \succ 0$, which means that, by equation~\eqref{eqn: QNPSD}, $\mathcal{Q}_N~\succeq~0$. 

\underline{Induction Step}: In the induction step, we assume that $\mathcal{Q}_{t+1} \succeq 0$ and show that $\mathcal{Q}_t \succeq 0$. 
The equation
\begin{align}
\begin{split}\label{eqn: qtPSDProof}
	\mathcal{Q}_t &= \mathcal{X}^T_t \mathcal{F}_t\mathcal{X}_t + \mathcal{Y}_t^T\mathcal{Q}_{t+1}\mathcal{Y}_t,
\end{split}
\end{align}
follows from algebraic manipulation of~\eqref{eqn: qtDef}. 
By definition, $\mathcal{F}_t\succ 0$, and, by the induction hypothesis, $\mathcal{Q}_{t+1} \succeq 0$. Thus, by equation~\eqref{eqn: qtPSDProof}, $\mathcal{Q}_t \succeq 0$. 
\end{proof}

\section{Proof of Lemma~\ref{lem: convexitySufficiency}}
\begin{proof}
We resort to induction.

\underline{Base Case}: In the base case, we show that $\mathcal{D}^T \mathcal{F}_{N-1} \mathcal{D} + \mathcal{B}^T \mathcal{Q}_N \mathcal{B}$ is positive definite. We first show that there exists $\mathcal{T}_N \succeq 0$ such that 
\begin{equation}\label{eqn: lem3BCFirstStep}
	\mathcal{Q}_N = \mathcal{H}^T Q_N \mathcal{H} + \mathcal{T}_N. 
\end{equation}
From the structure of $\mathcal{C}$ and $\mathcal{F}_t$, we have, for all $t = 0, \dots, N$, 
\begin{equation}\label{eqn: lemma3Proof1}
	\mathcal{C}^T \mathcal{F}_t \mathcal{C} = \mathcal{H}^T Q_t \mathcal{H} + \widetilde{\mathcal{C}}^T R_t \widetilde{\mathcal{C}}.
\end{equation}
Substituting~\eqref{eqn: lemma3Proof1} into~\eqref{eqn: QRTerminalConditions} and performing algebraic manipulations, we have
\begin{equation}
	\mathcal{Q}_N = \mathcal{H}^T Q_N \mathcal{H} + \widetilde{\mathcal{X}}_N^T R_N \widetilde{\mathcal{X}}_N,
\end{equation}
where
\begin{equation}
	\widetilde{\mathcal{X}}_N = \widetilde{\mathcal{C}} - \Psi \left(\Psi^T R_N \Psi \right)^\dagger \Psi^T R_N \widetilde{\mathcal{C}}.
\end{equation}
Since $R_N \succ 0$, $\mathcal{T}_N = \widetilde{\mathcal{X}}_N^T R_N \widetilde{\mathcal{X}}_N \succeq 0$. 

Proceeding, we note that \[\mathcal{B}^T \mathcal{H}^T = \left[\begin{array}{cccc} \Gamma + BLK\Psi & 0 & 0 & 0 \end{array}\right],\] which means that, for all $t = 0, \dots, N$,
\begin{equation}\label{eqn: lemma3Proof2}
	\mathcal{B}^T \mathcal{H}^T Q_t \mathcal{HB} = \left( \Gamma + BLK\Psi \right)^T Q_t \left(\Gamma + BLK\Psi\right).
\end{equation}
Then, substituting for~\eqref{eqn: lem3BCFirstStep},~\eqref{eqn: lemma3Proof2}, and the definitions of $\mathcal{B}, \mathcal{D}$, and $\mathcal{F}_{N-1}$, we have
\begin{equation}
	\mathcal{D}^T \mathcal{F}_{N-1} \mathcal{D} + \mathcal{B}^T \mathcal{Q}_N \mathcal{B} = \widetilde{\mathcal{B}}^T \mathcal{F}_{N-1} \widetilde{\mathcal{B}} + \mathcal{B}^T \mathcal{T}_N \mathcal{B},
\end{equation}
where $\widetilde{\mathcal{B}} = \left[\begin{array}{cc} \left(\Gamma + BLK\Psi\right)^T & \Psi^T \end{array} \right]^T$. We now show that $\widetilde{\mathcal{B}}$ has full column rank. By contradiction, suppose it does not have full column rank. Then $\exists \mu \neq 0$ such that
\begin{equation}\label{eqn: lemma3Proof3}
	\left[\begin{array}{c} \Gamma + BLK\Psi \\  \Psi \end{array} \right]\mu = 0. 
\end{equation}
Equation~\eqref{eqn: lemma3Proof3} implies that, for some $\mu \neq 0$, $\left[\begin{array}{c} \Gamma \\ \Psi \end{array} \right] \mu = 0$, which contradicts  Assumption~\ref{ass: GammaPsiInjective}. Since $\widetilde{\mathcal{B}}$ has full column rank, $\mathcal{F}_{N-1} \succ 0$ (by definition), and $\mathcal{T}_N \succeq 0$, equation~\eqref{eqn: lemma3Proof2} shows that $\mathcal{D}^T \mathcal{F}_{N-1} \mathcal{D} + \mathcal{B}^T \mathcal{Q}_N \mathcal{B} \succ 0$. 

\underline{Induction Step}:
In the induction step, we assume that $ \mathcal{D}^T \mathcal{F}_t \mathcal{D} + \mathcal{B}^T \mathcal{Q}_{t+1} \mathcal{B} \succ 0$ and show that $ \mathcal{D}^T \mathcal{F}_{t-1} \mathcal{D} + \mathcal{B}^T \mathcal{Q}_{t} \mathcal{B} \succ 0$. We first show that there exists $\mathcal{T}_t \succ 0$ such that
\begin{equation}\label{eqn: lem3ISFirstStep}
	\mathcal{Q}_t = \mathcal{H}^T Q_t \mathcal{H} + \mathcal{T}_t. 
\end{equation}
From equation~\eqref{eqn: qtPSDProof}, we have $\mathcal{Q}_{t} = \mathcal{X}_t^T \mathcal{F}_t\mathcal{X}_t + \mathcal{Y}_t^T\mathcal{Q}_{t+1}\mathcal{Y}_t$. Substituting equation~\eqref{eqn: lemma3Proof1} into~\eqref{eqn: qtPSDProof} and performing algebraic manipulations, we have
\begin{equation}
	\mathcal{Q}_t = \mathcal{H}^T Q_t \mathcal{H} + \widetilde{\mathcal{X}}_t^T R_t \widetilde{\mathcal{X}}_t + \mathcal{Y}_t^T \mathcal{Q}_{t+1} \mathcal{Y}_t,
\end{equation}
where
\begin{equation}
\begin{split}
	\widetilde{\mathcal{X}}_t =& \widetilde{\mathcal{C}} - \Psi \left( \mathcal{D}^T \mathcal{F}_t \mathcal{D} + \mathcal{B}^T \mathcal{Q}_{t+1} \mathcal{B} \right)^{-1} \left(\Psi^T R_t \widetilde{\mathcal{C}} + \mathcal{B}^T \mathcal{Q}_{t+1} \mathcal{A} \right).
\end{split}
\end{equation}
By Lemma~\ref{lem: QPSD}, $\mathcal{Q}_{t+1} \succeq 0$ for $t = N-1, \dots, 0$, and, by definition, $R_t\succ 0$. Thus, $\mathcal{T}_t = \widetilde{\mathcal{X}}_t R_t \widetilde{\mathcal{X}}_t + \mathcal{Y}_t^T \mathcal{Q}_{t+1} \mathcal{Y}_t \succeq 0$. Then, substituting for~\eqref{eqn: lemma3Proof2},~\eqref{eqn: lem3ISFirstStep}, and the definitions of $\mathcal{B}, \mathcal{D},$ and $\mathcal{F}_{t-1}$, we have
\begin{equation}\label{eqn: lem3Proof4}
	\mathcal{D}^T \mathcal{F}_{t-1} \mathcal{D} + \mathcal{B}^T \mathcal{Q}_{t} \mathcal{B} = \widetilde{\mathcal{B}}^T \mathcal{F}_{t-1} \widetilde{\mathcal{B}} + \mathcal{B}^T \mathcal{T}_t \mathcal{B},
\end{equation}
which shows that $\mathcal{D}^T \mathcal{F}_{t-1} \mathcal{D} + \mathcal{B}^T \mathcal{Q}_{t} \mathcal{B} \succ 0$. 
\end{proof}

\section{Proof of Lemma~\ref{lem: terminalOptimumExistence}}
\begin{proof}
We show that $\mathscr{N}\left(\mathcal{D}^T \mathcal{F}_N \mathcal{D}\right) \subseteq \mathscr{N}\left(\mathcal{F}_N\mathcal{D}\right)$. Note that $\mathcal{D}^T\mathcal{F}_N \mathcal{D} = \Psi^T R_N \Psi$, and $\mathscr{N} \left(\mathcal{F}_N \mathcal{D} \right) = \mathscr{N} \left(R_N \Psi \right)$.  Let $\mu \in \mathscr{N}\left(\Psi^T R_N \Psi\right)$, and, by contradiction, suppose $\mu \notin \mathscr{N}\left(R_N\Psi\right)$. Then $R_N\Psi\mu \neq 0$, which means that $\Psi \mu \neq 0$. Since $R_N \succ 0$ and $\Psi \mu \neq 0$, we have \[\mu^T\Psi^T R_N \Psi \mu \neq 0,\]
which means that \[\Psi^T R_N \Psi \mu \neq 0.\]
This is a contradiction since $\mu \in \mathscr{N}\left(\Psi^T R_N\Psi\right)$. Thus, we have $\mathscr{N}\left(\Psi^T R_N \Psi\right) \subseteq \mathscr{N}\left(R_N\Psi\right)$.

For any matrix $M$, $\mathscr{N}\left(M\right) = \mathscr{R}^\perp \left(M^T\right)$, where $\mathscr{R}^\perp$ denotes the orthogonal range space of a matrix. Thus, we have $\mathscr{R}^\perp\left(\mathcal{D}^T \mathcal{F}_N \mathcal{D} \right) \subseteq \mathscr{R}^\perp \left(\mathcal{D}^T \mathcal{F}_N\right)$, which means that $\mathscr{R}\left(\mathcal{D}^T \mathcal{F}_N \right) \subseteq \mathscr{R}\left(\mathcal{D}^T\mathcal{F}_N\mathcal{D}\right)$. By construction, $\mathcal{D}^T \mathcal{F}_N \left(\mathcal{C}\xi_N + \mathcal{M} \nu_N\right)$ belongs to $\mathscr{R}\left(\mathcal{D}^T \mathcal{F}_N\right)$, so it must also belong to $\mathscr{R}\left(\mathcal{D}^T \mathcal{F}_N \mathcal{D}\right)$. 
\end{proof}

\section{Proof of Lemma~\ref{lem: innovationEquality}}\label{sect: proofInnovationEquality}
From equation~\eqref{eqn: attackerFilter}, we have that the attacker's innovation is
\begin{equation}
	\nu_t = \widetilde{y}_t - C\left(A \widehat{x}_{t-1} + BL\left(\widehat{x}_{t-1}^0 + \omega_t\right) + \Gamma e_{t-1} \right). \label{eqn: innovationEquality5}
\end{equation}
By definition, the system's innovation when there is no attack is
\begin{equation}
	\nu_t^0 = y_t^0 - C\left( A \widehat{x}_{t-1}^0  + BL \widehat{x}_{t-1}^0 \right). 
\end{equation}
We compute $\widetilde{y}_t$ using the system in equation~\eqref{eqn: ssModel1}:
\begin{align}
	\widetilde{y}_t = y_t^0 + \sum_{j = 0}^{t-1} CA^{t-1-j} \left( \Gamma e_j + \omega_{j+1}\right). \label{eqn: innovationEquality4}
\end{align}
From the Kalman Filtering equations for both the system and attacker, we have
\begin{align}
	\widehat{x}_{t-1}^0 &= A^{t-1} \widehat{x}_0^0 + \sum_{j = 0}^{t-2} A^{t-2-j} \left( BL \widehat{x}_j ^0 + K\nu_{j+1}^0 \right), \label{eqn: innovationEquality1} \\
	\begin{split}
		\widehat{x}_{t-1} &= A^{t-1} \widehat{x}_0 +  \sum_{j = 0}^{t-2} A^{t-2-j} \Big( BL \left(\widehat{x}_j^0 + \omega_{j+1} \right)+ \label{eqn: innovationEquality2} K\nu_{j+1} + \Gamma e_j \Big). 
	\end{split}
\end{align}
At time $t=0$, the attack $e_0$ does not affect the internal state of the system $x_0$, and the attacker Kalman filter uses $\widetilde{y}_0$ to compute $\widehat{x}_0$. Since $\widetilde{y}_0$ is not affected by the attack $e_0$ (i.e., $\widetilde{y}_0 = y_0^0$), we have $\widehat{x}_0 = \widehat{x}_0^0$. Following equations~\eqref{eqn: innovationEquality1} and~\eqref{eqn: innovationEquality2}, we have
\begin{align}
\begin{split}\label{eqn: innovationEquality3}
	\widehat{x}_{t-1} =& \widehat{x}_{t-1}^0 + \sum_{j = 0}^{t-2} A^{t-2-j} \Big(\Gamma e_j + BL \omega_{j+1} + K \left( \nu_{j+1} - \nu_{j+1}^0 \right) \Big). 
\end{split}
\end{align}

Substituting equations~\eqref{eqn: innovationEquality4} and~\eqref{eqn: innovationEquality3} into equation~\eqref{eqn: innovationEquality5}, we have, after algebraic manipulation, 
\begin{equation}\label{eqn: innovationEquality6}
	\nu_t = \nu_t^0 + \sum_{j=0}^{t-2} CA^{t-2-j}K \left( \nu_{j+1} - \nu^0_{j+1} \right). 
\end{equation}
At time $t = 0$, both the attacker's and system's Kalman filter have the same prediction $\widehat{x}_{0|-1}$ because $e_t = 0$ for $t < 0$. Since $\widetilde{y}_0 = y_0^0$, we have $\nu_0 = \widetilde{y}_0 - C\widehat{x}_{0|-1} = \nu_0^0$. Changing the index of summation in equation~\eqref{eqn: innovationEquality6} and using the fact that $\nu_0 = \nu_0^0$, we have
\begin{equation}\label{eqn: innovationEquality7}
	\nu_t = \nu_t^0 + \sum_{j=0}^{t-1} CA^{t-1-j}K \left(\nu_{j} - \nu^0_{j} \right). 
\end{equation}
We now resort to induction to show that $\nu_t = \nu_t^0$. In the base case, we have already shown that $\nu_0 = \nu_0^0$. In the induction step, we assume that $\nu_j = \nu_j^0$ for $j = 0, \dots, t-1$. Then, by equation~\eqref{eqn: innovationEquality7}, we have $\nu_t = \nu_t^0$. 

\section{Proof of Lemma~\ref{lem: asymptoticIntermediate}}\label{sect: asymptoticLemma}
We first show that $\mathcal{W}_{t} \succeq  0$ for $t = 0, \dots, N+1$. We resort to induction.

\underline{Base Case}: In the base case, we have $\mathcal{W}_{N+1} = 0 \succeq 0$.
  
\underline{Induction Step}: In the induction step, we assume that $\mathcal{W}_{t+1} \succeq 0$ and show that $\mathcal{W}_t \succeq 0$. By~\eqref{eqn: WRecursion}, we have
	\begin{align}
		\begin{split}\label{eqn: WRecursion2}
			\mathcal{W}_t &= \mathcal{H}^T Q_t \mathcal{H} + \mathcal{A}^T \mathcal{W}_{t+1} \mathcal{A} - \mathcal{A}^T \mathcal{W}_{t+1} \mathcal{B} \left(\mathcal{B}^T \mathcal{W}_{t+1} \mathcal{B} \right)^\dagger \mathcal{B}^T \mathcal{W}_{t+1} \mathcal{A},
		\end{split}\\
		\begin{split}\label{eqn: WRecursion3}
			&= \mathcal{H}^T Q_t \mathcal{H} + \left(\mathcal{A} - \mathcal{B} \left(\mathcal{B}^T \mathcal{W}_{t+1}\mathcal{B}\right)^\dagger\mathcal{B}^T \mathcal{W}_{t+1} \mathcal{A} \right)^T\mathcal{W}_{t+1} \times \\ &\quad \left(\mathcal{A} - \mathcal{B} \left(\mathcal{B}^T \mathcal{W}_{t+1}\mathcal{B}\right)^\dagger\mathcal{B}^T \mathcal{W}_{t+1} \mathcal{A} \right)^T.
		\end{split}
	\end{align}
Since $Q_t \succ 0$ by definition and $\mathcal{W}_{t+1} \succeq 0$ by the induction hypothesis, we have $\mathcal{W}_t \succeq 0$. 

We now prove Lemma~\ref{lem: asymptoticIntermediate} by showing that $\mathscr{N}\left(\mathcal{B}^T\mathcal{W}_{t+1} \mathcal{B}\right) \subseteq \mathscr{N}\left(\mathcal{W}_{t+1} \mathcal{B} \right)$. Let \[\mu \in \mathscr{N}\left(\mathcal{B}^T\mathcal{W}_{t+1} \mathcal{B}\right).\] Then, we have
\begin{align}
	\mu^T \mathcal{B}^T \mathcal{W}_{t+1} \mathcal{B} \mu &= 0,\\
	\left\lVert \mathcal{W}_{t+1}^{\frac{1}{2}} \mathcal{B} \mu \right\rVert_2^2 &=0. \label{eqn: asymptoticIntermediateProof1}
\end{align}
The matrix $\mathcal{W}_{t+1}^{\frac{1}{2}}$ exists since $\mathcal{W}_{t+1} \succeq 0$. From the non-negativity property of norms, equation~\eqref{eqn: asymptoticIntermediateProof1} gives
\begin{equation}
	\mathcal{W}_{t+1}^{\frac{1}{2}} \mathcal{B}\mu = 0,
\end{equation}
which means $\mathcal{W}_{t+1} \mathcal{B} \mu = 0$ and $\mu \in \mathscr{N}\left(\mathcal{W}_{t+1} \mathcal{B} \right)$. 

For any matrix $M$, $\mathscr{N}(M) = \mathscr{R}\left(M^T\right)$, where $\mathscr{R}^\perp$ denotes the orthogonal space of a matrix. Thus, we have $\mathscr{R}^\perp\left(\mathcal{B}^T \mathcal{W}_{t+1} \mathcal{B}\right) \subseteq \mathscr{R}^\perp \left(\mathcal{B}^T \mathcal{W}_{t+1} \right)$, which means that $\mathscr{R} \left(\mathcal{B}^T \mathcal{W}_{t+1} \right) \subseteq \mathscr{R} \left(\mathcal{B}^T \mathcal{W}_{t+1} \mathcal{B}\right)$. By construction, for any $\xi_t \in \mathbb{R}^{6n}$, $\mathcal{B}^T \mathcal{W}_{t+1} \mathcal{A} \xi_t \in \mathscr{R}\left(\mathcal{B}^T \mathcal{W}_{t+1} \right)$, so we have $\mathcal{B}^T \mathcal{W}_{t+1} \mathcal{A} \xi_t \in \mathscr{R}\left(\mathcal{B}^T \mathcal{W}_{t+1}\mathcal{B} \right).$
\bibliography{References}
\end{document}